\documentclass[12pt]{amsart}

\usepackage{graphicx, verbatim,amsmath,amssymb, float}
\usepackage{hyperref}
\hypersetup{
    colorlinks=true,
    linkcolor=blue,
    filecolor=magenta,      
    urlcolor=cyan,
}

\usepackage{amsthm}
\usepackage{amsfonts}
\usepackage{cleveref}
\usepackage{comment}
\usepackage{braket}

\newcommand{\tdens}{\tau}
\newcommand{\edens}{\varepsilon}
\newcommand{\bT}{{\bar\tau}}
\DeclareMathOperator{\Tr}{Tr}

\newcommand{\R}{\mathbb{R}}

\newcommand{\calG}{\mathcal{G}}

\newcommand{\calO}{\mathcal{O}}

\newcommand{\be}{\begin{equation}}
\newcommand{\ee}{\end{equation}}

\newcommand{\inr}[2]{\langle #1, #2 \rangle}

\newtheorem{lemma}{Lemma}

\newtheorem{theorem}[lemma]{Theorem}
\newtheorem{corollary}[lemma]{Corollary}

\newtheorem{proposition}[lemma]{Proposition}
\newtheorem{remark}[lemma]{Remark}

\title{Typical large graphs with given edge and triangle densities}

\date{\today}

\author{
Joe Neeman
\and Charles Radin
\and Lorenzo Sadun
}
\address{Joe Neeman\\Department of Mathematics\\The University of
  Texas at Austin\\ Austin, TX 78712} \email{joeneeman@gmail.com}
\address{Charles Radin\\Department of Mathematics\\The University of
  Texas at Austin\\ Austin, TX 78712} \email{radin@math.utexas.edu}
\address{Lorenzo Sadun\\Department of Mathematics\\The University of
  Texas at Austin\\ Austin, TX 78712} \email{sadun@math.utexas.edu}

\thanks{
This work was partially supported by the Deutsche Forschungsgemeinschaft (DFG, German Research
Foundation) under Germany's Excellence Strategy – EXC-2047/1 – 390685813,
and by a fellowship from the Alfred P. Sloan Foundation.
}

\begin{document}

\begin{abstract}
The analysis of large simple graphs
with extreme values of the densities of edges and triangles has been
extended to the statistical structure of typical graphs of fixed
intermediate densities, by the use of large deviations of
Erd\H{o}s-R\'enyi graphs. We prove that the typical graph exhibits
sharp singularities as the constraining densities vary between different curves of
extreme values, and we determine the precise nature of the
singularities. The extension to graphs with fixed densities of edges
and $k$-cycles for odd $k>3$ is straightforward and we note the simple changes in
the proof.
 
\end{abstract}

\maketitle

\section{Introduction}\label{sec:Intro}

Our results concern the nature of simple graphs on $n$ vertices, for
large $n$, constrained to have density $\edens$ of edges and $\tdens$
of triangles. The range of achievable values of the pair $(\edens,\tdens)$ was an
old problem in extremal combinatorics initiated by Tur\'an in 1941
\cite{Tur}. The extremal graph theory of these constraints was
recently completed by Razborov {\em et al} in \cite{Raz, PR}, which also
contain a good history of this problem; see
Figure~\ref{FIG:phase_space}. The graphs associated with some parts of
the boundary of this region are not unique, but it is not
difficult to characterize those probabilistically using the graphon
formalism of Borgs et al.~\cite{BCL,BCLSV} and Lov\'{a}sz et
al.~\cite{LS1,LS2,LS3,Lov}. See Section 4 in~\cite{RS1} for a
discussion of those extremal graphs in terms relevant to this work.

The boundary of the parameter space
depicted in Figure~\ref{FIG:phase_space} falls naturally into three
curves: the upper boundary $\tdens=\edens^{3/2}$, 
the line segment on which $\tdens=0$,
and the scalloped curve completed by Razborov {\em et al}. On the 
upper boundary $\tdens>\edens^3$, while on the latter two curves
$\tdens<\edens^3$. 
The nature of the graphs associated with the points
on each curve is similar, but those associated with different curves
are not~\cite{RS1}.

\begin{figure}[ht]
\centering
\includegraphics[angle=0,width=0.8\textwidth]{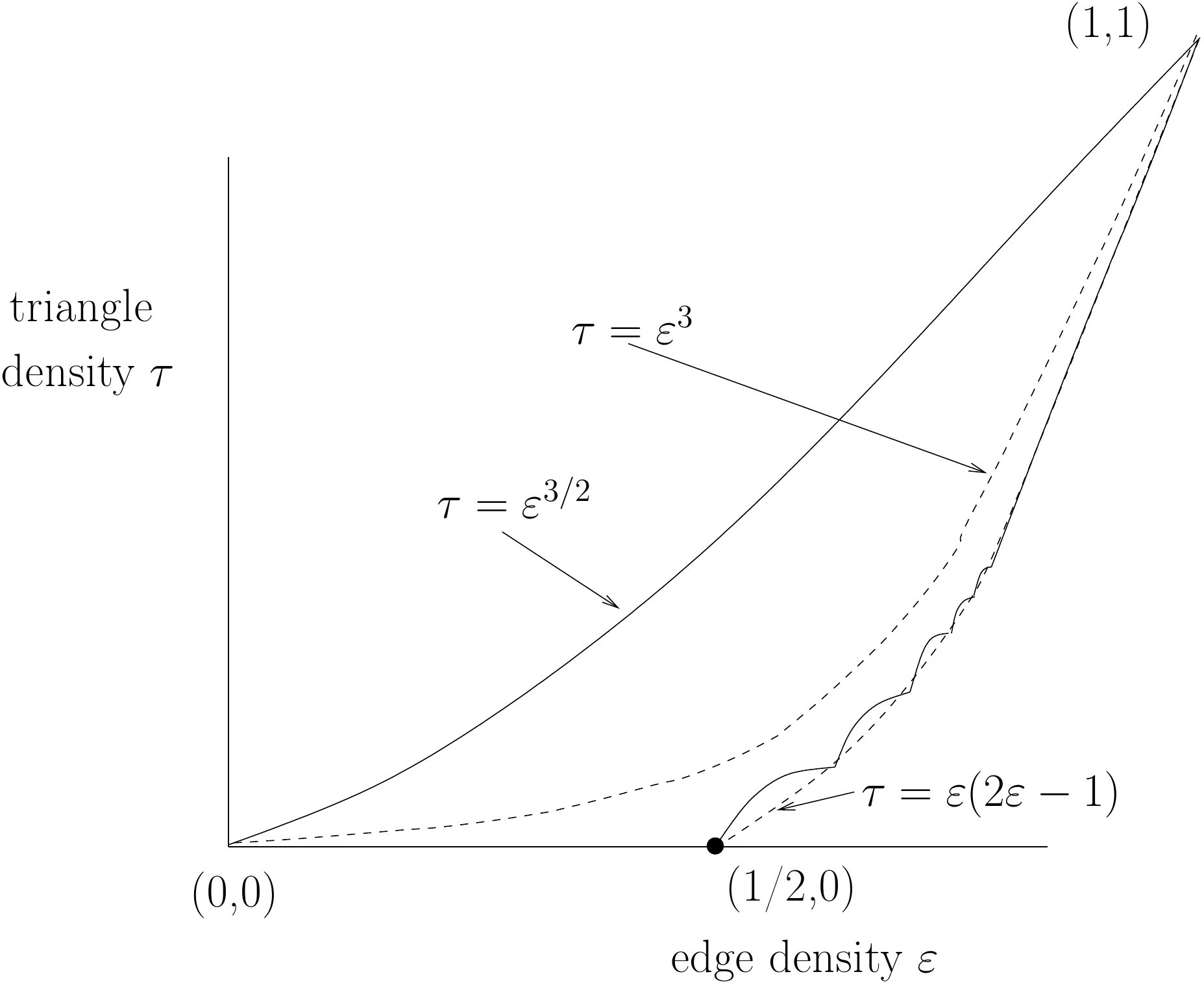}
\caption{The boundary of the achievable parameters  is in
  solid lines. The figure is distorted to expose features.}
\label{FIG:phase_space}
\end{figure}

\begin{figure}[ht]
\centering
\includegraphics[angle=0,width=0.8\textwidth]{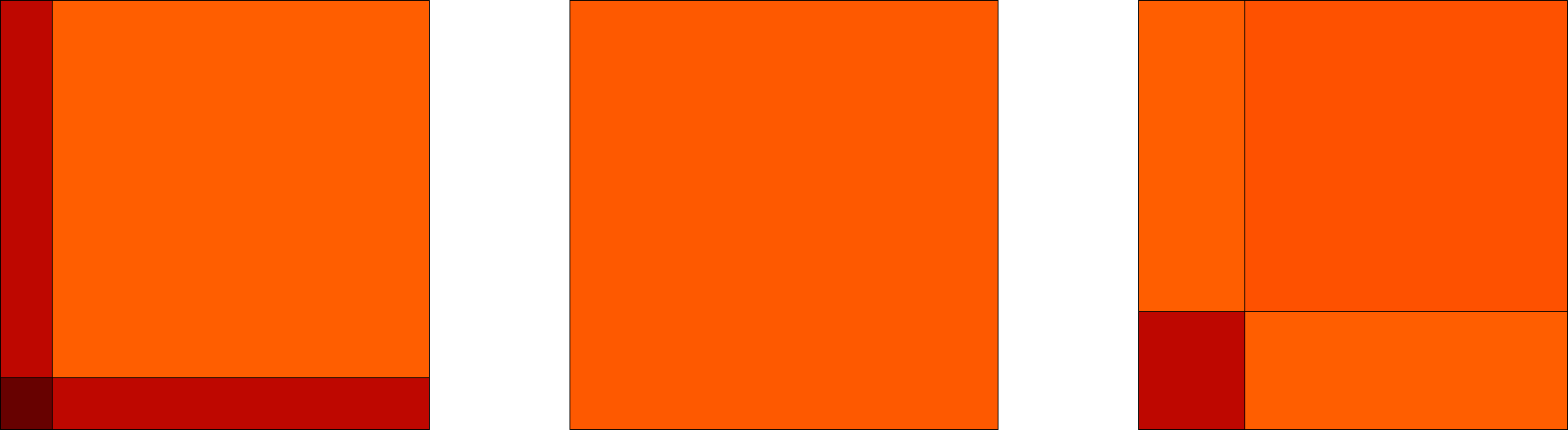}
\caption{Optimal graphons (not to scale) at $\edens > \frac 12$ and at $\tdens = \edens^3 + \delta^3$ on the left,
  $\tdens=\edens^3$ in the middle, and
  $\tdens=\edens^3 - \delta^3$  on the right. On the left, the strip at the bottom has height $O(\delta^3)$; on the right it has height $O(\delta)$.}
\label{FIG:graphons}
\end{figure}

In this paper we analyze the statistical
structure of `typical' graphs with
constraints on $(\edens,\tdens)$ in the {\em interior} of the
parameter space of Figure~\ref{FIG:phase_space}.  We use the graphon
formalism to describe asymptotic probabilistic structure, and use
the rate function for certain large deviations of Erd\H{o}s-R\'enyi
graphs to interpret typicality, a notion central to our analysis.  We
give a careful discussion of typicality in
Section~\ref{sec:asymptotics}, but informally it means `all but exponentially
few' graphs with the given density constraints, exponential in the
number of vertices of the system.

Given the known graphons associated with the boundary, we focus on how
a typical density-constrained graph changes as the densities
$(\edens,\tdens)$ move from one of the three basic components of the
boundary to another, in particular as $(\edens,\tdens)$ moves between
$\tdens=\edens^{3/2}$ and each of the other two curves: the line
segment and the scallops. The statistical structure of a graph typical for
constraints $(\edens,\tdens)$ is easily computed from the
`entropy-optimal graphon' associated by large deviations theory with $(\edens,\tdens)$, but we
emphasize that as an element in any such study one must establish
values of $(\edens,\tdens)$ at which there is a {\em unique}
entropy-optimal graphon, since without uniqueness one cannot really
speak of `typical' behavior associated with $(\edens,\tdens)$. In fact,
we have found that proving such uniqueness has been the most difficult
part of the analysis, requiring stringent a priori knowledge of the
possible optimal graphons. (There is numerical evidence of constraints
$(\edens,\tdens)$ supporting multiple entropy-optimal graphons; see
for instance the discussion of discontinuous transitions
in~\cite{KRRS3}. There is also a proof, Theorem 5.1 in \cite{KRRS}, of
nonunique entropy-optimal graphons in the related model with fixed
densities of edges and 2-stars.)

Our main result is the explicit determination, for any fixed
$1/2<\edens<1$, of the unique entropy-optimizing graphon associated
with $(\edens,\tdens)$ as $\tdens$ crosses through
$\tdens=\edens^3$, and the corollary that the behavior is singular at
$\tdens=\edens^3$. For a qualitative picture of the singularity see
Figure~\ref{FIG:graphons}.
In terms of large graphs, for these values of $(\edens, \tdens)$ we
approximately determine the number of graphs with edge density approximately
$\edens$ and triangle density approximately $\tdens$, and we show that most
such graphs have a specific structure. For more details
about the connection between entropy-optimizing graphons and large graphs,
see Section~\ref{sec:asymptotics}.

To describe our results more quantitatively we need some notation
(with more detail in Section~\ref{sec:asymptotics}). A
graphon $g(x,y)$ on $[0,1]\times [0,1]$ is called {\em bipodal} if
there is a decomposition of $[0,1]$ into 2 intervals (`vertex
clusters') $C_1$ and $C_2$ and constants $a,b,d$ such that
\begin{eqnarray}\label{def:abd}
g(x,y)&=a & \hbox{if }(x,y)\in C_1\times C_1\cr
g(x,y)&=b & \hbox{if }(x,y)\in C_2\times C_2 \cr 
g(x,y)&=d & \hbox{if }(x,y)\in C_1\times C_2 \hbox{ or }(x,y)\in C_2\times C_1.
\end{eqnarray}
We denote the length of $C_1$ by $c$.

It is immediate that graphs with independent edges satisfy
$\tdens=\edens^3$; see Figure~\ref{FIG:phase_space}. It was previously
proven \cite{KRRS2} for $0<\edens<1/2$ and for $1/2 < \edens <1$, that
the entropy-optimal graphon for $\tdens$ slightly greater than
$\edens^3$ is unique and bipodal and the structure was determined. In
this paper we consider the more difficult case of the graphons with
$\tdens$ slightly less than $\edens^3$. For $1/2 < \edens <1$ we prove
that it again is unique and bipodal, and again determine its
structure. We also determine the asymptotic behavior of the entropy as
$(\edens,\tdens)$ approaches the curve $\tdens=\edens^3,\
1/2<\edens<1$ from above and from below.  Our main results can be
summarized as follows, using the function:

\begin{equation}\label{eq:H-def}
H(p) = -[p \ln(p) + (1-p) \ln(1-p)].
\end{equation}

\begin{theorem}\label{thm:b11}
There is an open subset ${\calO}_1$ in the planar set of achievable parameters $(\edens,\tdens)$,
whose upper boundary is the curve $\tdens = \edens^3,\  1/2<\edens<1$, such that
at $(\edens,\tdens)$ in ${\calO}_1$ there is a unique entropy-optimizing
graphon $g_{(\edens,\tdens)}$. This graphon is bipodal and for fixed $(\edens,\tdens)=(e,e^3-\delta^3)$, the values of 
$a,b,c,d$ can be approximated to arbitrary accuracy via an explicit iterative scheme. These parameters can also be expressed via asymptotic power series in $\delta$ whose leading terms
are:
\begin{eqnarray}\label{approximation:b11-abcd}
a & = & 1-e - \delta + O(\delta^2) \cr 
b & = & e - \frac{\delta^2}{2e-1} + O(\delta^3) \cr 
c & = & \frac{\delta}{2e-1} - \frac{2\delta^2}{2e-1} + O(\delta^3) \cr 
d & = & e + \delta + \frac{\delta^2}{eH'(e)} \left ( H'(e) - \left (e-\frac12 \right ) H''(e) \right ) + O(\delta^3). \end{eqnarray}

\end{theorem}

\begin{theorem}\label{thm:f11}
There is an open subset ${\calO}_2$ in the planar set of achievable parameters $(\edens,\tdens)$,
whose lower boundary is the curve $\tdens = \edens^3,\  1/2<\edens<1$, such that
at each $(\edens,\tdens)$ in ${\calO}_2$ there is a unique entropy-optimizing
graphon $g_{(\edens,\tdens)}$. This graphon is bipodal and for fixed $(\edens,\tdens)=(e,e^3+\Delta \tau)$ the values of 
$a,b,c,d$ can be approximated to arbitrary accuracy via an explicit iterative scheme. These parameters can also be expressed via asymptotic power series in $\Delta \tau$ whose leading terms are:
\begin{eqnarray}\label{approximation:f11-abcd}
a & = & a_0 + O(\Delta \tau) \cr 
b & = & e - \frac{2 \Delta \tau}{3e(2e-1)} + O(\Delta \tau^2) \cr 
c & = & \frac{\Delta \tau}{3e(2e-1)^2} + O(\Delta \tau^2) \cr 
d & = &  1-e + O(\Delta \tau), 
\end{eqnarray}
where $a_0$ is the solution to 
\begin{equation}\label{a-equ}
H'(a_0) = \left ( 1 - \frac{2}{e} \right ) H'(e).
\end{equation}
\end{theorem}


\begin{theorem}\label{thm:entropy}

The entropy function $s(\edens,\tdens)$ is real analytic in the variables $\edens$ and
$\tdens$ in the two subsets ${\calO}_2$ and ${\calO}_1$ of Theorems \ref{thm:b11} and \ref{thm:f11}, 
 which share the boundary curve $\tdens=\edens^3$. 
The entropy $s(\edens,\tdens)$ is continuous on their common boundary but is
not differentiable at any point of the curve. 
As $\tdens$ approaches $\edens^3$ from above and below, 
the following estimates apply:
\begin{itemize}
\item For fixed $\edens=e \ne 1/2$, as $\tdens=e^3+\Delta \tau$ 
approaches $e^3$ from above, 
\begin{eqnarray}
s(e,e^3\!+\!\Delta \tau) & = & H(e) - \frac{2\Delta \tau}{3e(1-2e)} H'(e)  \cr
&&+ \frac{\Delta
  \tau^2}{9e^2(1\!-\!2e)^4}
\left ( H(a_0) \!-\! H(e) \!+\! H'(e) \left ( 3(a_0\!-\!e) +\frac{2(1\!-\!2e)}{e}(a_0\!+\!3e\!-\!2)
\right ) \right )\cr
&& + \frac{2 \Delta \tau^2 H''(e)}{9e^2(1-2e)^2} + O(\Delta \tau^3). 
\end{eqnarray}
\item For fixed $\edens= e \in (1/2,1)$, as $\tdens = e^3 - \delta^3$ approaches $e^3$ from below, 
\begin{eqnarray}
s(e,e^3-\delta^3) &=&  H(e) + \frac{\delta^2}{2e-1}H'(e) - \frac{2\delta^3 \nu }{(2e-1)^2} \cr 
&& + \delta^4 \left ( \frac{H'''(e)}{3(2e-1)} + \frac{4 \nu}{(2e-1)^3}
  - \frac{2\nu^2}{e H'(e) (2e-1)^2} \right )\cr
&& + O(\delta^5),
\end{eqnarray}
where $\nu = H'(e) - \left ( e - \frac12 \right ) H''(e).$ 
\end{itemize}
In particular, $\partial s/\partial \tdens$ diverges as $\delta^{-1}$ as $\tdens$ approaches 
$e^3$ from below, as previously shown in \cite{RS2}. As $\tdens$ approaches $e^3$ 
from above, $\partial s/\partial \tdens$ does {\em not} diverge, instead approaching 
the finite (negative) value ${2H'(e)}/{3e(2e-1)}$. 
Furthermore, the second derivative $\partial^2 s/\partial \tdens^2$ is
negative. 
\end{theorem}

Theorem \ref{thm:b11} and the second half of Theorem \ref{thm:entropy} generalize to models where we fix the densities
of edges and $k$-cycles, where $k$ is odd, instead of edges and 
triangles. The problem actually gets progressively easier as $k$
increases, insofar as our concentration of degree estimates 
become sharper. Let $\tdens_k$ denote the density of $k$-cycles. 

\begin{theorem}\label{thm:kgon} Let $k>3$ be an odd integer and let $1/2 < e < 1$. 
\begin{itemize}
\item For sufficiently small $\delta>0$, the entropy-maximizing graphon
subject to the constraints $\edens=e$ and $\tdens=e^k-\delta^k$
is bipodal with parameters
\begin{eqnarray}\label{approximation:kgon-abcd}
a & = & 1-e - \delta + O(\delta^2) \cr 
b & = & e - \frac{\delta^2}{2e-1} + O(\delta^3) \cr 
c & = & \frac{\delta}{2e-1} - \frac{2\delta^2}{2e-1} + O(\delta^3) \cr 
d & = & e + \delta + O(\delta^{k-1}). 
\end{eqnarray}
The entropy is 
\begin{eqnarray}
s(e,e^k-\delta^k) &=&  H(e) + \frac{\delta^2}{2e-1}H'(e) 
- \frac{2\delta^3 \nu }{(2e-1)^2} \cr 
&& + \delta^4 \left ( \frac{4 \nu}{(2e-1)^3} + \frac{H'''(e)}{3(2e-1)} \right ) + O(\delta^5).
\end{eqnarray}
\item When $\Delta \tau$ is sufficiently small
and positive, the optimizing graphon 
with $\edens=e$ and $\tdens_k=e^k+\Delta\tau$
is bipodal with 
\begin{eqnarray}\label{approximation:f11-kgon}
a & = & a_0 + O(\Delta \tau) \cr 
b & = & e - \frac{2 \Delta \tau}{ke^{k-2}(2e-1)} + O(\Delta \tau^2) \cr 
c & = & \frac{\Delta \tau}{ke^{k-2}(2e-1)^2} + O(\Delta \tau^2) \cr 
d & = &  1-e + O(\Delta \tau), 
\end{eqnarray}
where $a_0$ is the solution to (\ref{a-equ}).
The entropy is 
\begin{equation}\label{eq:entropy-kgon}
s(e,e^3\!+\!\Delta \tau)  =  H(e) - \frac{2\Delta \tau}{ke^{k-2}(1-2e)} H'(e) + O(\Delta \tau^2). 
\end{equation}
\end{itemize}
\end{theorem} 

Most of Theorem~\ref{thm:f11} was already proven in
\cite{KRRS2}; here we merely sharpen the estimates. The analysis for Theorem~\ref{thm:b11}, dealing with
$\tdens<\edens^3$, is the main part of this paper. A substantial portion of our analysis is
devoted to proving that any optimal
graphon for $\tdens$ less than but sufficiently close to $\edens^3$,
with $1/2<\edens<1$, is bipodal. Our argument is quite different from
that of~\cite{KRRS2}, which dealt with the case $\tdens>\edens^3$,
because two of the main techniques in that paper do not apply to undersaturated
graphs. Specifically,~\cite{KRRS2} begins with multipodality for
graphs oversaturated with 2-stars, but there are no graphs that are
\emph{undersaturated} with 2-stars.  Then~\cite{KRRS2} repeatedly
applies the Euler-Lagrange equations; but besides the fact that it seems
challenging to rigorously establish Euler-Langrange equations without
first knowing multipodality, the Lagrange multipliers for
undersaturated graphs are expected to explode as $\delta \to
0$. Finally, there is evidence of a different nature indicating why
the situation for $\tdens<\edens^3$ is more complicated than for
$\tdens>\edens^3$. The graphs associated with the boundary curve
$\tdens=\edens^{3/2}$ are all similar, and \cite{KRRS2} proves the
same is true for $\tdens$ just above $\edens^3$,  for all
$0<\edens<1/2$ and $1/2 <\edens<1$. But in~\cite{RS2} it
is proven that for $\edens=1/2$ and any $0<\tdens<(1/2)^3$ there is a
unique optimal graphon, bipodal and similar in kind to those on the
boundary curve where $\tdens=0$, but quite {\em dissimilar} to what we
have proven for $\edens>1/2$. In other words, the two different
boundary curves with $\tdens < \edens^3$ -- the part with $\tdens=0$
and the scallops -- generate qualitatively different paths for crossing
$\tdens=\edens^3$, and so far we have found the case
$0<\edens<1/2$ impervious to our techniques.

In Section \ref{sec:kgon} we prove Theorem \ref{thm:kgon}. 
The proof follows the same steps as the proofs of Theorems
\ref{thm:b11} and \ref{thm:entropy}, only with sharper bounds 
on the parameter $\mu$. As a result, many of the most difficult 
steps in the proof of Theorem \ref{thm:b11} can be streamlined
or avoided entirely. Instead of repeating all the calculations,
we merely note the changes needed to adapt the proof for 
triangles to higher values of $k$. 

One motivation for our study of extremal graphs is an
older problem in extremal combinatorics, the densest packing of
spheres. In two and three dimensional Euclidean space the densest
packing of congruent spheres has been proven, the latter by a
celebrated tour de force by Hales et al~\cite{Ha}. The densest
packings have volume fraction $\pi/\sqrt{12}\approx 0.91$ in two
dimensions and $\pi/\sqrt{18}\approx 0.74$ in three dimensions and in
both cases the optimal packings are highly ordered. It has been an
important open problem for many years to prove that a typical packing
at fixed volume fraction, in both dimensions, loses its order at a
sharp value of volume fraction below the optimum, though this has not
yet been proven~\cite{Low}. The mathematical setup for sphere packings
is similar to the one we are using for graphs, with typicality defined
through the entropy, the rate function of a large deviation
variational principle~\cite{Lan,Ell}. In sphere packing the singularity is known
as a phase transition.

\section{Regularity in the Edge/Triangle System}\label{sec:regularity}

\subsection{Notation for Asymptotics}\label{sec:asymptotics}

We consider a simple graph $G$ (undirected,
with no multiple edges or loops) with a vertex set $V(G)$ of labeled
vertices. For a subgraph $K$ of $G$, let $T_K(G)$ be the number of
maps from $V(K)$ into $V(G)$ that send edges to edges. The \emph{density}
$\tdens_K(G)$ of $K$ in $G$ is then defined to be
$$
\tdens_K(G):= \frac{|T_K(G)|}{n^{|V(K)|}},
$$
where $n = |V(G)|$. An important special case is where $K$ is a triangle.
We use the letters $e$ and $t$ to denote specific values of the edge density
$\edens$ and the triangle density $\tdens$. 
For $\alpha > 0$ and $\bT=(e,t)$ we define
$\displaystyle Z^{n,\alpha}_{\bT}$ to be the number of graphs $G$ on $n$
vertices satisfying
$$
\edens(G) \in (e-\alpha,e+\alpha), \ \tdens(G) \in (t-\alpha,t+\alpha).
$$

Define the \emph{(constrained) entropy} $s({\bT})$ to be the
exponential rate of growth of $Z^{n,\alpha}_{\bT}$ as a function of $n$:
$$
s({\bT})=\lim_{\alpha\searrow 0}\lim_{n\to \infty}\frac{\ln(Z^{n,\alpha}_{\bT})}{ n^2}.
$$
The double limit defining the entropy $s({\bT})$ is known to
exist~\cite{RS1}. To analyze it we make use of a variational
characterization of $s({\bT})$, and for this we need further notation to
analyze limits of graphs as $n\to \infty$. (This work was recently
developed in \cite{BCL,BCLSV,LS1,LS2,LS3}; see also the recent book
\cite{Lov}.)  The (symmetric) adjacency matrices of graphs on $n$
vertices are replaced, in this formalism, by symmetric, measurable
functions $g:[0,1]^2\to[0,1]$; the former are recovered by using a
partition of $[0,1]$ into $n$ consecutive subintervals. The functions
$g$ are called {\em graphons}.

For a graphon $g$ define the \emph{degree function} $d(x)$ to be
$d(x)=\int^1_0 g(x,y)dy$.  
The triangle density of a graphon $g$ is $\tdens(g) = \int g(x,y) g(y,z) g(z,x)\, dx\, dy\, dz$,
and the edge density is $\edens(g) = \int g(x,y)\, dx\, dy$.
The \emph{entropy} of a graphon $g$ is
$S(g) = \int\int H(g(x,y))\, dx\,dy$, where $H$ is defined as in~\eqref{eq:H-def}.

The following result is Theorem 3.1 in~\cite{RS1}, itself a special case of Theorem 4.1 in~\cite{RS2}:
\begin{theorem}[A Variational Principle]\label{thm:variational}
For any values $\bT= (e, t)$ in the parameter space
we have 
$s({\bT}) = \max [S(g)]$, where $S$ is 
maximized over all graphons $g$ with $\edens(g)=e$ and $\tdens(g)=t$.
\end{theorem}
\noindent 
The existence of an entropy-maximizing 
graphon $g=g_{\bT}$ for any pair $\bT$ of possible densities 
was proven in \cite{RS1}, again adapting a proof in \cite{CV}.

We consider two graphs {\it equivalent} if they are obtained from one
another by relabeling the vertices. For graphons, the analogous
operation is applying a measure-preserving map $\psi$ of $[0,1]$ into
itself, replacing $g(x,y)$ with $g(\psi(x),\psi(y))$, see~\cite{Lov}.
The equivalence classes of graphons under relabeling are called
\emph{reduced graphons}, and 
graphons are equivalent
if and only if they have the same subgraph densities for all possible
finite subgraphs \cite{Lov}. In the remaining sections of the paper,
whenever we claim that a graphon has a property (e.g., uniqueness as an 
entropy maximizer), the caveat ``up to relabeling'' is implied. 

Density-constrained graphons that maximize $S$, which we call
`entropy maximizing graphons', were introduced in~\cite{RS1} and have
been studied slowly but steadily ever since. They can tell us what
`most' or `typical' large density-constrained graphs are like: if
$g_{\bT}$ is the only reduced graphon maximizing $S$ with
$\bT(g)=\bT$, then as the number $n$ of vertices diverges and
$\alpha_n\to 0$, all but exponentially few graphs with densities $\bT_i(G)\in
(\tdens_i-\alpha_n,\tdens_i+\alpha_n)$ will have reduced graphons close
to $g_{\bT}$~\cite{RS1}.  This is based on large deviations of
Erd\H{o}s-R\'enyi graphs from~\cite{CV}. We emphasize that this
interpretation requires that the maximizer be unique; this has been
difficult to prove in most cases of interest, is responsible for the
slow advance of the study of `typical' density-constrained graphs, and
is an important focus of this work.

\subsection{Related work}

Recent developments in probabilistic combinatorics have dramatically expanded
the scope of results like Theorem~\ref{thm:variational} from~\cite{RS1,RS2}: the number of graphs
with given edge and subgraph counts can be approximated by the solutions to
certain entropic maximization problems~\cite{CD,Augeri,HMS,Eldan,CD16,CD20}. These results are particularly
challenging for sparse graphs (i.e.\ for graphs with $n$ vertices and $m_n =
p_n \binom{n}{2}$ edges for $p_n \to 0$).
On the other hand, it is quite rare that these entropic maximization problems
can be solved explicitly. We only know of one other such case where the optimizers
are non-constant graphons: for the upper
tail of sparse random graphs, the maximization problem was solved by~\cite{BGLZ}.
Thanks to even more recent results in the theory of large deviations~\cite{CD,Augeri,HMS},
it is now known
that for $1 \gg p_n \gg (\log^2 n)/{\sqrt n}$ and any fixed $u > 0$, a
random graph $G$ with $n$ vertices and $m_n = p_n \binom{n}{2}$ edges satisfies
\[
    \Pr\left(\tau(G) \ge (1 + u) p_n^3 \binom{n}{3}\right) = \exp\left(-(1+o(1))n^2 p_n^2 \log \frac{1}{p_n} \min \left\{\frac{u^{2/3}}{2}, \frac{u}{3}\right\}\right).
\]
Moreover, graphs with bipodal structure saturate the bound on the right
hand side (\cite{CD20} call them ``clique'' or ``hub'' graphs depending on
whether ${u^{2/3}}/{2}$ or $u/3$ is smaller). Some structural results in
the sparse case are also known:~\cite{HMS} show that conditioned on
having triangle density $\tau(G)$ at least $(1 + u) p_n^3 \binom{n}{3}$, it is likely that
the graph contains a large clique, a large hub, or both.

\section{Strategy of proof}\label{sec:strategy}

As discussed above, the proof of Theorem~\ref{thm:b11} involves two
very different sorts of arguments.
First, we must show that any optimal graphon is bipodal. Then we must
solve
the finite-dimensional optimality problem for bipodal graphons.
The details are complicated, with many technical estimates, so we 
present an overview of the argument here. 

We start the proof that the optimal graphon in bipodal in  Section~\ref{sec:Initial_approximations}. 
We begin by
computing an upper bound for the entropy,
based on the fact 
that any graphon with $\edens=e$ and $\tdens= e^3-\delta^3$ must have 
\begin{equation}\label{eq:L^2-lower} 
\iint \left [g(x,y)-e \right ]^2 \, dx\, dy \ge \delta^2,
\end{equation}
and from the maximum possible entropy of any graphon satisfying 
(\ref{eq:L^2-lower}). We then exhibit an explicit ``model'' 
bipodal graphon that comes within 
$O(\delta^3)$ of achieving the upper bound. 
(This is where the assumption that $e>1/2$ comes in. The 
same upper bound applies when $e < 1/2$, but isn't nearly
as sharp.) Since~\eqref{eq:L^2-lower} is sharp only when $g(x,y) - e$
has rank 1, we conclude that $g - e$ is close to rank 1
and also that it concentrates mainly on two values.
This shows that $g$ is close in $L^2$ to a bipodal
graphon: there are well-defined quadrants of the unit square on which $g$
has $L^2$-small fluctuations.

To show that any optimal graphon is bipodal, we assume that it isn't and
then construct explicit competitors by first averaging $g$ on each quadrant
(which maintains the edge density while possibly changing the triangle
density), and then making small adjustments to some parameters to recover
the original triangle density. We aim to 
show that if $g$ wasn't bipodal to begin with,
it would be possible to increase the entropy with such a perturbation.

This step requires estimates on the best bipodal graphon. 
The space of bipodal graphons is only 4-dimensional, so maximizing the entropy
becomes a problem in 4-variable calculus, which we tackle 
in Section \ref{sec:best-bipodal}. 
We use the constraints on 
$\varepsilon$ and $\tau$ to eliminate two of the variables, writing
the entropy as a function of the value $a$ of the graphon in 
one quadrant and a parameter $\mu$ that measures how far the 
degrees are from being constant. Taking derivatives of the entropy 
with respect to $a$ and $\mu$, and setting them equal to zero, 
yields the estimates in Theorem \ref{thm:b11}. 

This analysis is complicated by the fact that we do not know, 
a priori, that the parameters can be expressed as power series 
in $\delta$.
When using a Taylor series
to approximate values of the function $H(p)$ near $p=e$ or $p=1-e$,
or when estimating the quantity $\frac{c \Delta a}{1-c}$ in terms
of $\mu$, it is not immediately clear which terms must be kept 
and which can be ignored. We get around this with a bootstrap, 
using initial estimates to establish which terms must be kept, and
then using the revised expansions to get more accurate estimates. 
In particular, we use the concentration-of-degree estimate 
from Section~\ref{sec:Initial_approximations}  to claim that $\mu = 
O(\delta^{3/2})$, which we then use to prove that $\mu$ is in 
fact $O(\delta^2)$. Aside from that concentration-of-degree 
estimate, this part of the proof is completely independent of 
the proof that the optimizing graphon is bipodal. 

In Subsection \ref{sec:unique} we turn 
to the uniqueness of the optimizing bipodal graphon. 
We compute the Hessian of the entropy
with respect to $a$ and $\mu$, obtaining a matrix of the form 
$$\begin{pmatrix} K_1 \delta^2+O(\delta^3) & O(\delta^2) 
\cr O(\delta^2) & K_2 \delta^{-1} + O(1) \end{pmatrix}, $$
for negative constants $K_1$ and $K_2$ (which depend on $e$ but not on $\delta$).
Thus the equations for critical points 
of $S$, subject to the constraints (\ref{approximation:b11-abcd}), are 
well-approximated by a non-singular linear problem, which of course
has a unique solution. 

Having established the properties of the optimal bipodal graphon, 
we return in Sections \ref{sec:averaging} and \ref{sec:improvement}
to showing that the optimal graphon is in fact bipodal. In Section \ref{sec:averaging} we assume that the optimizing graphon is not 
bipodal and 
we estimate the change in the entropy, and the change in the triangle
count, from averaging the optimal graphon over each quadrant
to obtain a bipodal graphon. In Section \ref{sec:improvement} we use 
the results of Section \ref{sec:best-bipodal} to show that the parameters 
of this bipodal graphon can then be perturbed to recover the original
triangle count, with higher than the original entropy. Since the original
graphon was assumed to be optimal, this is a contradiction, implying that
the original optimal graphon was in fact already bipodal. 

In other words, the problem of finding the optimal graphon reduces to 
the finite-dimensional problem of finding the optimal {\em bipodal} 
graphon, which we already solved in Section \ref{sec:best-bipodal}.
The entropy function $S$ is analytic in the parameters
$(a,b,c,d)$ and the constraints $(\varepsilon,\tau) = (e, 
e^3-\delta^3)$ are analytic (actually algebraic) in 
$(a,b,c,d,e,\delta)$. This implies that the set of critical points
is a 2-dimensional analytic variety in $\R^6$. Away from
singularities, the implicit function theorem says that 
$(a,b,c,d)$ are analytic functions of $(e,\delta)$. The analysis of
Section \ref{sec:unique} shows that there are no singularities 
in the region defined by (\ref{approximation:b11-abcd}) where the actual 
optimal graphon lives, so the parameters of the optimal graphon,
and the entropy, 
are analytic in $(e,\delta)$, and therefore in $(\edens,\tdens)$, for $\tdens$ 
strictly less than, and sufficiently close to, $\edens^3$. That
completes the proof of Theorem \ref{thm:b11}. 

\section{Initial approximation}
\label{sec:Initial_approximations}

\subsection{Notation}

Working at a specific edge density $\edens=e$ between $\frac12$ and 1, 
we write $g$ for the graphon $g(x, y)$ and also $g_x(y) = g(x,y)$, and define 
$\Delta g(x,y) = g(x,y)-e$. 
We consider $\tdens = e^3 - \delta^3$ for sufficiently
small $\delta$ (depending on $e$). We take $g$ to be a maximizer of $S(g)$ subject to $\edens(g) = e$
and $\tdens(g) = e^3-\delta^3$.
In our asymptotic notation, we treat $e$ as fixed and consider $\delta \to 0$. That is,
the hidden constants in $O(\delta)$ are allowed to depend on $e$, but on 
nothing else.

Define $D(p) = p \ln({p}/{e}) + (1-p) \ln[({1-p})/({1-e})]$.
We will write $\|\cdot\|_2$ for either the $L^2$ norm on $[0, 1]$ (with respect to Lebesgue measure)
or the $L^2$ norm on $[0, 1]^2$ (with respect to the Lebesgue measure). It should be clear from
the context which of these is the case.
For a function $h \in L^2([0, 1]^2)$, we write $T_h$ for the integral operator with kernel $h$:
\[
    (T_h u) (x) = \int_y h(x,y) u(y)\, dy,
\]
which is compact and Hilbert-Schmidt because $h \in L^2$.

The first step towards proving that $g$ is bipodal is to show that $g$ is well approximated in $L^2$ by a bipodal graphon of a certain form.
Specifically, we show:
\begin{proposition}\label{prop:approx-B11}
    There is a function $v: [0, 1] \to \R$ such that
    \begin{enumerate}
        \item $v$ takes only two values, one of them $\sqrt{2e-1} \pm O(\delta)$, and one of them $\pm O(\delta)$;
            \label{it:approx-B11:values}
        \item $\int_0^1 v(x)\, dx = 0$;
            \label{it:approx-B11-mean}
        \item $\|v\|_2^2 = \delta + O(\delta^2)$;
            \label{it:approx-B11-norm}
        \item $\|\Delta g(x,y) + v(x) v(y)\|_2^2 = \|g(x,y) - (e - v(x) v(y))\|_2^2 = O(\delta^3)$;
            \label{it:approx-B11-residual}
        \item if $C_2 = \{x: v(x) = O(\delta)\}$ and $C_1 = [0, 1] \setminus C_2$
            then for every $x \in C_2$,
            \begin{equation}\label{eq:every-vertex-Omega1}
                \|g_x - (1-e)\|_{C_1}^2 \ge \|g_x - e\|_{C_1}^2 - O(\delta^2);
            \end{equation}
            and
            \label{it:approx-B11-Omega1}
        \item for every $x \in C_1$,
            \begin{equation}\label{eq:every-vertex-Omega2}
                \|g_x - (1-e)\|_{C_1}^2 \le \|g_x - e\|_{C_1}^2 + O(\delta^2).
            \end{equation}
            \label{it:approx-B11-Omega2}
    \end{enumerate}
\end{proposition}

The point here is that $e - v(x) v(y)$ is a bipodal graphon with triangle
density $e - \|v\|_2^6 = e - \delta^3 \pm O(\delta^4)$; \Cref{prop:approx-B11} shows that
$g$ is close to this bipodal graphon in $L^2$. The regions $C_2$ and $C_1$ are called {\em podes}. $C_1$ is
the small pode and $C_2$ is the big pode, and the last two points of \Cref{prop:approx-B11}
show that no $x \in [0, 1]$ is classified into a clearly wrong pode.

\subsection{Entropy cost}
Our first step towards \Cref{prop:approx-B11} is a pretty good estimate (at least for small $\delta$) for the entropy cost of reducing
the triangle density.

Define
\begin{equation}\label{eq:C-def}
    C(e) = \frac{\ln \frac{e}{1-e}}{2e-1}.
\end{equation}

The relevance of $C(e)$ is that it characterizes the optimal trade-off, in some sense, between entropy
and $L^2$ mass.
\begin{lemma}\label{lem:C-prop}
    \[
        C(e) = \inf_{p \in [0, 1]} \frac{D(p)}{(p - e)^2}.
    \]
    Moreover, the infimum above is uniquely attained at $p = 1-e$, and it is a second-order minimum
    in the sense that for any $e \in (0, 1)$ there is a constant
    $c(e) > 0$ such that for any $p \in [0, 1]$,
    \[
        \frac{D(p)}{(p - e)^2} \ge C(e) + c(e) (p - (1-e))^2.
    \]
\end{lemma}

The proof of Lemma~\ref{lem:C-prop} can be found in~\cite{NRS}.

Recall that $H(e)$ is the entropy of the constant graphon. By the concavity of $H$, $S(g) < H(e)$; the following
lemma gives a bound on just how much smaller it must be.

\begin{lemma}\label{lem:entropy-cost}
    \[
        C(e) \delta^2 \le C(e) \|\Delta g\|_2^2 \le H(e) - S(g)
        \le C(e) \delta^2 + O(\delta^3).
    \]
\end{lemma}

\begin{proof}
    Note that $\int D(g) = H(e) - \int H(g) = H(e) - S(g)$. By \Cref{lem:C-prop},
    \[
        \int D(g) \ge \frac{D(1-e)}{(2e-1)^2} \int (g - e)^2 = C(e) \|\Delta g\|_2^2.
    \]
    Moreover, $\Tr[(T_{\Delta g})^3] = \tdens(g) - e^3 - 3e \int(d(x) - e)^2\, dx \le \tdens(g) - e^3 = -\delta^3$,
    and Cauchy-Schwarz implies that $\delta^3 \le -\Tr[(T_{\Delta g})^3] \le \|\Delta g\|_2^3$.
    This proves the first two claimed inequalities.

    To prove the final inequality, we construct a graphon having triangle density $e - \delta^3$
    and entropy cost $C(e) \delta^2 + O(\delta^3)$, and then the inequality follows from the
    fact that $g$ has minimal entropy cost among graphons with triangle density $e - \delta^3$.
Let
$v(x)$ be a function taking the value $-\sqrt{2e - 1}$ on a set of measure $a$,
and the value $\sqrt{2e-1}\, a/(1-a)$ on a set of measure $1-a$.
Then $\int v = 0$ and $\int v^2 = (2e - 1) {a}/({1-a})$.
For the graphon $h(x,y) = e - v(x) v(y)$ (which has edge density $e$),
the fact that our perturbation is rank-1 (and orthogonal to constants) implies that
\[
    \tdens(h) - e^3 = -\|v\|_2^6 = -(2e-1)^3 \frac{a^3}{(1-a)^3}.
\]
Now we fix $a$ so that $\tdens(h) = e^3 - \delta^3$.
Then $a = {\delta}/{(2e-1)} + O(\delta^2)$ as $\delta \to 0$, and
\[
    H(e) - S(h) = \int D(h) = a^2 D(1-e) + (1-a)^2 D(e - \Theta(a^2)) + 2a(1-a) D(e + \Theta(a)).
\]
Since $D(e) = D'(e) = 0$ and $D$ is twice differentiable, we conclude that
\[
    H(e) - S(h) = a^2 D(1-e) + O(a^3) = a^2 (2e-1) \ln \frac{e}{1-e} + O(a^3)
    = C(e)\delta^2 + O(\delta^3).
\]
Recalling that $S(g) \ge S(h)$, this completes the proof of the claimed upper bound.
\end{proof}

\subsection{Closeness to ideal values}

We saw that $\|\Delta g\|_2^2 = \|g - e\|_2^2 = O(\delta^2)$, but we can get a better
bound if we look at the distance of $g$ from \emph{either} $e$ or $1-e$.
A useful interpretation of this is that most of the $L^2$ mass of $\Delta g$ is
spent at values near $1-2e$.
This notion -- that most of the $L^2$ mass of something is spent near some particular
value -- will be used repeatedly. We will therefore study some basic properties of this notion.

Let
\begin{equation}\label{eq:V-def}
    V_a(x) = \min\{x^2, (x-a)^2\}.
\end{equation}
The point of this definition is that ``most of the $L^2$ mass
of $u$ is near $a$'' can be encoded as $\int V_a(u) \ll \int u^2$.
The basic homogeneity property of $V_a$ is that for any $a, x \in \R$, $V_a(x) = a^2 V_1(x/a)$.
This means that it mostly suffices to study properties of $V_1$.

Next, we show two stability properties: the notion of mass concentration is stable under small
perturbations of the function $u$, and also under small changes to the ideal value $a$.

\begin{lemma}\label{lem:V-perturbation}
    For any $u, w \in L_2(\mu)$, $\int V_1(u + w)\, d\mu \le 2 \int V_1(u)\, d\mu + 2 \|w\|_{L^2(\mu)}^2$.
\end{lemma}

\begin{proof}
    If $u^2(x) \le (u(x) - 1)^2$ then $(u(x) + w(x))^2 \le 2 u^2(x) + 2 w^2(x) = 2 V_1(u(x)) + 2 w^2(x)$.
    Similarly, if $(u(x) - 1)^2 \le u^2(x)$ then $(u(x) - 1 + w(x))^2 \le 2(u(x) - 1)^2 + 2 w^2(x) = 2 V_1(u(x)) + 2 w^2(x)$. Taking the minimum of these two inequalities, $V_1(u(x) + w(x)) \le 2 V_1(u(x)) + 2 w^2(x)$, and the
    claim follows by integrating.
\end{proof}

\begin{lemma}\label{lem:V-value-change}
    For any $u \in L_2(\mu)$ and any $\eta \ge 0$,
    \[
        \int V_{1 + \eta}(u)\, d\mu \le \int V_1(u)\, d\mu + 4\eta(1+\eta) \|u\|_{L^2(\mu)}^2.
    \]
\end{lemma}

\begin{proof}
    If $u \le 1/2$ then $V_{1 +\eta}(u) = V_1(u) = u^2$.
    On the other hand, if $u \ge 1/2$ then $V_{1 + \eta}(u) \le (u - 1 - \eta)^2 \le (u - 1)^2 + \eta^2 + \eta = V_1(u) + \eta^2 + \eta$. Markov's inequality implies that $\mu\{u \ge 1/2\} \le 4 \|u\|_{L^2(\mu)}^2$,
    and so
    \[
        \int V_{1 + \eta}(u) \, d\mu \le \int V_1(u) + 1_{\{u \ge \frac 12\}} (\eta^2 + \eta)\, d\mu
        \le \int V_1(u)\, d\mu + 4 \eta(1 + \eta) \|u\|_{L^2(\mu)}^2.
    \]
\end{proof}

Our final property of $V$ will allow us to show that if $u(x) + u(y)$ puts most of its
mass near $1$, then the same is true of $u$.

\begin{lemma}\label{lem:V-product}
    Let $u$ be a function on a finite measure space $(\Omega, \mu)$.
    If
    $\|u\|_{L^2(\mu)}^2 \le \mu(\Omega)/32$
    and
    \[
        \int_{\Omega \times \Omega} V_1(u(x) + u(y))\, d\mu^{\otimes 2} \le K
    \]
    for some $0 < K < \mu(\Omega)^2/256$ then
    \[
        \int_{\Omega} V_1(u(x))\, d\mu \le \frac{C K}{\mu(\Omega)}
    \]
    for some universal constant $C$.
\end{lemma}

Note that the restriction $\|u\|_{L^2(\mu)}^2 \le \mu(\Omega)/32$ is required to rule
out the situation where $u$ is close to the constant function $1/2$, in which case our
desired conclusion wouldn't hold.

\begin{proof}
    We can assume without loss of generality that $\mu$ is a probability measure.
    Let $A = \{x: u(x) \le 1/4\}$. By Markov's inequality and the fact that $\int_\Omega u^2 \, d\mu \le {1}/{32}$,
    $\mu(A) \ge 1/2$. On $A \times A$, $u(x) + u(y) \le 1/2$ and so
    $V_1(u(x) + u(y)) = (u(x) + u(y))^2$. Hence,
    \[
        \int_{A \times A} V_1(u(x) + u(y))\, d\mu^{\otimes 2} = \int_{A \times A} (u(x) + u(y))^2 \ge 2\mu(A) \int_A u^2(x)\, d\mu.
    \]
    On the other hand,
    \[
        \int_{A \times A} V_1(u(x) + u(y)) \, d\mu^{\otimes 2} \le K,
    \]
    and so
    \begin{equation}\label{eq:V-of-sum-bound-on-A}
        \int_A V_1(u(x)) \, d\mu = \int_A u^2(x) \, d\mu \le \frac{K}{2\mu(A)} \le K.
    \end{equation}

    Since $K \le 1/{256}$, applying Markov's inequality to $u 1_A$ gives
    \[
        \mu(\{|u 1_A| \ge \frac 18\})
        \le 64 \int_A u^2\, d\mu \le 64 K \le \frac 14,
    \]
    meaning that $\mu(\{|u| \le 1/8\}) \ge 1/4$; let $B = \{x: |u(x)| \le 1/8\}$.
    For $y \not \in A$ and $x \in B$, $u(x) + u(y) \ge 1/4 - 1/8 =  1/8$, meaning that
    $(u(x) + u(y) - 1)^2 \le 64 V_1(u(x) + u(y))$. Hence,
    \[
        \int_{B \times A^c} (u(x) + u(y) - 1)^2 \,d\mu(x) \, d\mu(y) \le 64 \int_{\Omega \times \Omega} V_1(u(x) + u(y))\, d\mu(x)\, d\mu(y) \le 64 K.
    \]
    On the other hand, Cauchy-Schwarz gives $(u(x) + u(y) - 1)^2 \ge  (u(y) - 1)^2/2 - u^2(x)$, and so
    \begin{align*}
        \int_{B \times A^c} (u(x) + u(y) - 1)^2 \,d\mu(x) \, d\mu(y)
        &\ge \frac 12 \mu(B) \int_{A^c} (u(y) - 1)^2\, d\mu(y) - \mu(A^c) \int_B u^2(x)\, d\mu(x) \\
        &\ge \frac 18 \int_{A^c} (u(y) - 1)^2\, d\mu(y) - K,
    \end{align*}
    where we used~\eqref{eq:V-of-sum-bound-on-A} for the last inequality, noting that $B \subseteq A$.
    Rearranging, we have
    \[
        \frac 18 \int_{A^c} (u(y) - 1)^2\, d\mu(y) \le 65 K.
    \]
    Since $V_1(u(y)) \le 16 (u(y) - 1)^2$ for $y \not \in A$, this shows that
    \begin{equation}\label{eq:V-of-sum-bound-off-A}
        \int_{A^c} V_1(u(y))\, d\mu(y) \le 520 K.
    \end{equation}
    Combined with~\eqref{eq:V-of-sum-bound-on-A}, this completes the proof.
\end{proof}

We return to studying the perturbations of our graphon $g$. In \Cref{lem:entropy-cost}, we
saw that the most entropy-efficient way to perturb $e$ in $L^2$ was to set some of the values
of $g$
to $1-e$, which is equivalent to having $\Delta g$ equal to $1 - 2e$ at some points. We strengthen
this by showing that
most of the mass of $\Delta g$ must be spent near $1 - 2e$.

\begin{lemma}\label{lem:ideal-values}
    \[
        \int V_{1-2e}(\Delta g(x,y))\, dx\, dy = O(\delta^3).
    \]
\end{lemma}

\begin{proof}
    By the second part of \Cref{lem:C-prop},
    \begin{align*}
        H(e) - S(g) = \int D(g)
        &= \int \frac{D(g)}{(g - e)^2} (g - e)^2 \\
        &\ge  C(e) \int (1 + \Theta((g - (1-e))^2)) (g-e)^2 \\
        &= C(e) \|\Delta g\|_2^2 + \Theta\left(\int (g - (1-e))^2 (g-e)^2\right) \\
        &= C(e) \|\Delta g\|_2^2 + \Theta\left(\int V_{1-2e}(\Delta g)\right).
    \end{align*}
    On the other hand, \Cref{lem:entropy-cost} implies that
    \[
        H(e) - S(g) \le C(e) \delta^2 + O(\delta^3)
        \le C(e) \|\Delta g\|_2^2 + O(\delta^3),
    \]
    and comparing this to the previous bound proves the claim.
\end{proof}

\subsection{Concentration of degrees}

We define the ``degree'' of $x \in [0, 1]$ to be $d(x) = \int g(x,y)\, dy$.
Note that $\int d(x)\, dx = \edens(g) = e$. It turns out that having a non-constant
degree function increases the triangle density, so our optimal graphon $g$
must have an almost-constant degree function.

\begin{lemma}\label{lem:degree-concentration}
    \[
        \int (d(x) - e)^2 \, dx = O(\delta^4)
    \]
\end{lemma}

Recalling that $\delta^4 = \Theta(\|\Delta g\|_2^4)$,
this is better than the trivial bound (coming from Jensen's inequality) of
    \[
        \int (d(x) - e)^2 \, dx \le \|\Delta g\|_2^2.
    \]

\begin{proof}
    Start by observing that
    \[
        \Tr[(T_{\Delta g})^3] = t - e^3 + 3 e^3 - 3e \int d^2(x)\, dx
        = - \delta^3 + 3e \int (d(x) - e)^2\, dx.
    \]
    Cauchy-Schwarz gives $\Tr[(T_{\Delta g})^3] \ge -\|\Delta g\|_2^3$,
    and so
    \[
        \delta^3 \le \|\Delta g\|_2^3 - 3e \int (d(x) - e)^2\, dx.
    \]
    By the concavity of the function $t \mapsto t^{2/3}$, if $s < t$ then
    $(t - s)^{2/3} \le t^{2/3} - \frac 23 t^{-1/3} s$. Therefore,
    \[
        \delta^2 \le \|\Delta g\|_2^2 - \frac{2e \int (d(x) - e)^2\, dx}{\|\Delta g\|_2}.
    \]
    Comparing this to \Cref{lem:entropy-cost} gives
    \begin{multline*}
        C(e) \|\Delta g\|_2^2
        \le H(e) - S(g) \le C(e) \delta^2  + O(\delta^3) \\
        \le C(e) \|\Delta g\|_2^2 - 2eC(e) \frac{\int (d(x) - e)^2\, dx}{\|\Delta g\|_2} + O(\delta^3),
    \end{multline*}
    and so we conclude that
    \[
        \int (d(x) - e)^2\, dx \le O(\|\Delta g\|_2 \delta^3) = O(\delta^4).
    \]
\end{proof}

\subsection{Rank}

In this section, we will prove \Cref{prop:approx-B11}.
We'll start by just considering an eigenfunction
(which will not necessarily take only two values). Later, we'll round it.
\begin{lemma}\label{lem:rank-1}
    There is a function $\tilde v(x)$ such that
    \[
        \|\Delta g(x,y) + \tilde v(x) \tilde v(y)\|_2^2 = O(\delta^3).
    \]
\end{lemma}

\begin{proof}
    Recall that
    \begin{equation}\label{eq:rank-1-entropy-comparison}
        H(e) - S(g) \ge C(e) \|\Delta g\|_2^2,
    \end{equation}
    and then we used the fact that (if $\lambda_i$ are the eigenvalues of $T_{\Delta g}$)
    $\sum_i \lambda_i^2 \ge (\sum_i |\lambda_i|^3)^{2/3}$ to compare
    this to $\delta^3$. We can sharpen this eigenvalue comparison: if we write
    the eigenvalues $\lambda_i$ so that their absolute values are non-increasing,
    and if $\epsilon > 0$ is chosen so that
    $\sum_{i \ge 2} \lambda_i^2 = \epsilon \lambda_1^2$, then
    \[
        \|\lambda\|_3^3 \le \|\lambda\|_\infty \|\lambda\|_2^2 = |\lambda_1| \|\lambda\|_2^2
        \le \frac{\|\lambda\|_2^3}{\sqrt{1+\epsilon}}
    \]
    and so
    \[
        \|\Delta g\|_2^2 = \|\lambda\|_2^2  \ge (1 + \epsilon)^{1/3} \|\lambda\|_3^2.
    \]
    Recalling from Lemma~\ref{lem:degree-concentration} that
    $\|\lambda\|_3^3 \ge -\Tr[(T_{\Delta g})^3] = \delta^3 + 3e \|d - e\|_2^2$
    and that $\|d - e\|_2^2 = O(\delta^4)$, we have
    \[
        \|\Delta g\|_2^2 \ge (1 + \epsilon)^{1/3} (\delta^3 - O(\delta^4))^{2/3}
        = \delta^2 + \Omega(\epsilon \delta^2) - O(\delta^3).
    \]
    Combining this estimate with~\eqref{eq:rank-1-entropy-comparison} gives
    \[
        H(e) - S(g) \ge C(e) \delta^3 + \Omega(\epsilon \delta^2) - O(\delta^3).
    \]
    Compared to \Cref{lem:entropy-cost}, this shows that $\epsilon = O(\delta)$.
    In other words, we have $\sum_{i \ge 2} \lambda_i^2 = O(\delta \lambda_1^2)$.

    On the other hand, $\sum_{i \ge 1} \lambda_i^2 = \|\Delta g\|_2^2 = \Theta(\delta^2)$,
    and so $\lambda_1^2 = \Theta(\delta^2)$ and $\sum_{i \ge 2} \lambda_i^2 = O(\delta^3)$.
    In particular, if $u(x)$ is an eigenfunction of $\Delta g$ with eigenvalue $\lambda_1$,
    normalized so that $\|u\|_2 = 1$, then
    \[
        \|\Delta g - \lambda_1 u(x) u(y)\|_2^2 = O(\delta^3).
    \]
    Finally, note that $\lambda_1 < 0$, because
    $\sum_i \lambda_i^3 = t - e^3 + O(\delta^4) = -\delta^3 + O(\delta^4)$,
    and since
    $\sum_{i \ge 2} |\lambda_i|^3 \le \left(\sum_{i \ge 2} \lambda_i^2\right)^{3/2} = O(\delta^{9/2})$,
    we must have $\lambda_1^3 = -\delta^3 + O(\delta^4)$.
    Setting $\tilde v(x) = \sqrt{|\lambda_1|} u(x)$ completes the proof.
\end{proof}

From now on, we fix a function $\tilde v$ satisfying \Cref{lem:rank-1}.
The following bound just comes from combining Lemma~\ref{lem:rank-1} with Lemma~\ref{lem:ideal-values}
and the triangle inequality (in the form of \Cref{lem:V-perturbation}).
\begin{corollary}\label{cor:rank-1-values}
    \[
        \int V_{2e-1}(\tilde v(x) \tilde v(y))\, dx\, dy = O(\delta^3).
    \]
\end{corollary}

Our next goal is to show that we can replace $\tilde v(x)$ by a rounded version.
We'll start by ignoring the sign of $\tilde v$.

\begin{lemma}\label{lem:rounding-v-without-sign}
    Let $\bar v(x)$ be either $0$ or $\sqrt{2e-1}$, whichever is closer to $|\tilde v(x)|$.
    There is a universal constant $C$ such that
    \[
        \|\bar v(x) \bar v(y) - |\tilde v(x) \tilde v(y)|\|_2^2 \le C \int V_{2e-1}(\tilde v(x) \tilde v(y))\, dx\, dy.
    \]
\end{lemma}

\begin{proof}
    Let $c = \sqrt{2e-1}$, and let $w(x,y)$ be either $0$ or $c^2$, whichever
    is closer to $|\tilde v(x) \tilde v(y)|$.
    Since $\bar v(x) \bar v(y)$ is always either $0$ or $c^2$, we have the pointwise bound
    \[
        \Big||\tilde v(x) \tilde v(y)| - w(x,y)\Big| \le \Big||\tilde v(x) \tilde v(y)| - \bar v(x) \bar v(y)\Big|.
    \]
    Our first goal is to show the reverse inequality for most points $x$ and $y$:
    \begin{equation}\label{eq:pointwise-reverse}
        \Big||\tilde v(x) \tilde v(y)| - \bar v(x) \bar v(y)\Big|
        \le
        C \Big||\tilde v(x) \tilde v(y)| - w(x,y)\Big| = C V_{2e-1}(\tilde v(x) \tilde v(y)).
    \end{equation}
    Let
    \begin{align*}
        A_1 &= \{x: |\tilde v(x)| \le \frac 12 c\} = \{\bar v = 0\} \\
        A_2 &= \{x: |\tilde v(x)| \ge \frac 43 c\}.
    \end{align*}
If $x \in A_1$ then $\bar v(x) \bar v(y) = 0$ no matter the value of $y$. Now if $x \in A_1$
and $|\tilde v(y)| \le c$ then $w(x,y) = 0$ and so $\bar v(x) \bar v(y) = w(x,y)$. If $x \in A_1$
and $c < |\tilde v(y)| \le 4c/3$ then $w(x,y) = c^2$ but $|\tilde v(x) \tilde v(y)| \le  2c^2/3$,
meaning that
\[
    \Big||\tilde v(x) \tilde v(y)| - w(x,y)\Big| \ge \frac{c^2}{3} \ge \frac 12 |\tilde v(x) \tilde v(y)| = \frac 12 \Big|\bar v(x) \bar v(y) - |\tilde v(x) \tilde v(y)|\Big|
\]
To summarize: if $x \in A_1$ and $y \not \in A_2$ then~\eqref{eq:pointwise-reverse}
holds with $C = 2$. Of course, the same holds if $y \in A_1$ and $x \not \in A_2$.

If $x, y \not \in A_1$ then $\bar v(x) \bar v(y) = c^2$ and $w(x,y)$ could be either $0$ or $c^2$.
If $w(x,y) = c^2$ then $\bar v(x) \bar v(y) = w(x,y)$ and so~\eqref{eq:pointwise-reverse}
holds with $C=1$; while if $w(x,y) = 0$
then
\begin{multline*}
    \Big||\tilde v(x) \tilde v(y)| - w(x,y)\Big| = |\tilde v(x) \tilde v(y)| \\
    \ge \frac 18 (|\tilde v(x) \tilde v(y)| + c^2)) \ge \frac 18 \Big| |\tilde v(x) \tilde v(y)| - \bar v(x) \bar v(y)\Big|,
\end{multline*}
and so~\eqref{eq:pointwise-reverse} holds with $C=8$.

The only case where we have \emph{not} shown~\eqref{eq:pointwise-reverse} is when $x \in A_1$
and $y \in A_2$ (or vice versa). In particular, by integrating out~\eqref{eq:pointwise-reverse}
on every set where we have proven it, we get
\begin{equation}\label{eq:most-of-the-integral}
    \int_{[0,1]^2 \setminus ((A_1 \times A_2) \cup (A_2 \times A_1))} (|\tilde v(x) \tilde v(y)| - \bar v(x) \bar v(y))^2\, dx\, dy \le
    C \||\tilde v(x) \tilde v(y)| - w\|_2^2.
\end{equation}

Finally, we consider the case $x \in A_1$, $y \in A_2$: here the pointwise bound~\eqref{eq:pointwise-reverse}
is not necessarily true, and so we give an integral bound instead.
Note that
\begin{align*}
    \||\tilde v(x) \tilde v(y)| - w\|_2^2
    &\ge \int_{A_1 \times A_1} (|\tilde v(x) \tilde v(y)| - w(x,y))^2 \, dx\, dy  \\
    &= \int_{A_1 \times A_1} \tilde v^2(x) \tilde v^2(y)\, dx\, dy \\
    &= \left(\int_{A_1} \tilde v^2(x)\, dx\right)^2
\end{align*}
and similarly
\begin{align*}
    \||\tilde v(x) \tilde v(y) - w\|_2^2
    &\ge \int_{A_2 \times A_2} (|\tilde v(x) \tilde v(y)| - w(x,y))^2 \, dx\, dy  \\
    &= \int_{A_2 \times A_2} (|\tilde v(x) \tilde v(y)| - c^2)^2\, dx\, dy \\
    &\ge \int_{A_2 \times A_2} (|\tilde v(x) \tilde v(y)| - \frac{9}{16} |\tilde v(x)| |\tilde v(y)|)^2\, dx\, dy \\
    &\ge \frac{1}{16} \int_{A_2 \times A_2} \tilde v^2(x) \tilde v^2(y) \, dx\, dy \\
    &= \left(\frac{1}{4} \int_{A_2} \tilde v^2(x) \, dx\right)^2.
\end{align*}
So then
\begin{align*}
    \int_{A_1 \times A_2} (|\tilde v(x) \tilde v(y)| - \bar v(x) \bar v(y))^2\, dx\, dy
    &= \int_{A_1 \times A_2} \tilde v^2(x) \tilde v^2(y) \, dx\, dy \\
    &= \int_{A_1} \tilde v^2(x)\, dx \int_{A_2} \tilde v^2(y)\, dy \\
    &\le 4 \| |\tilde v(x) \tilde v(y)| - w\|_2^2.
\end{align*}
Of course, we have the same bound if we replace $A_1 \times A_2$ with $A_2 \times A_1$.
Together with~\eqref{eq:most-of-the-integral}, this shows that
\[
    \||\tilde v(x) \tilde v(y)| - \bar v(x) \bar v(y)\|_2^2 \le C \||\tilde v(x) \tilde v(y)| - w\|_2^2.
\]
This is almost the same as the claim; the difference is that the right hand side above is
\[
    C \int \min\{\tilde v^2(x) \tilde v^2(y), (|\tilde v(x) \tilde v(y)| - c^2)^2\}\, dx\, dy,
\]
whereas right hand side in the claim has no absolute values.
But since $\Big||v(x) v(y)| - c^2\Big| \le |v(x)v(y) - c^2|$, the claim follows.
\end{proof}

Next, we handle the signs. Of course, $\tilde v(x)$ can be negated without
changing $\tilde v(x) \tilde v(y)$, but the rounding to $\{0, \sqrt{2e-1}\}$ is
affected by the sign. Therefore, we may need to replace $\tilde v$ by $- \tilde v$
in order to give bounds for the rounded version.

\begin{lemma}\label{lem:rounding-v}
    After possibly replacing $\tilde v$ by $-\tilde v$, the following holds.
    Let $v(x)$ be either $0$ or $\sqrt{2e-1}$, whichever is closer to $\tilde v(x)$. Then
    \[
        \|v(x) v(y) - \tilde v(x) \tilde v(y)\|_2^2 \le
        C \int V_{2e-1}(\tilde v(x) \tilde v(y))\, dx\, dy.
    \]
\end{lemma}

\begin{proof}
    Let $c = \sqrt{2e-1}$ and
    recall the definition of $\bar v$ from Lemma~\ref{lem:rounding-v-without-sign}:
    $\bar v(x)$ is either $0$ or $c$, whichever is closer to $|\tilde v(x)|$.
    In particular,
    \begin{align*}
        v(x) v(y) &= \bar v(x) \bar v(y) & &\text{ if $\tilde v(x) \ge -\frac c2$ and $\tilde v(y) \ge -\frac c2$} \\
        v(x) v(y) &= \bar v(x) \bar v(y) = 0 & &\text{ if $|\tilde v(x)| \le \frac c2$ or $|\tilde v(y)| \le \frac c2$} \\
        v(x) v(y) &= 0 \ne c^2 = \bar v(x) \bar v(y) & &\text{ if $\tilde v(x) < -\frac c2$ and $|\tilde v(y)| > \frac c2$, or vice versa}.
    \end{align*}

    Let $A_1 = \{x: \tilde v(x) < - c/2\}$ and $A_2 = \{x: \tilde v(x) >  c/2\}$, so that
    $v(x) v(y) \ne \bar v(x) \bar v(y)$ only on $A_1 \times A_1$, $A_1 \times A_2$, and $A_2 \times A_1$. Hence,
    \begin{align*}
        &\|\tilde v(x) \tilde v(y) - v(x) v(y)\|_2^2 - \|\tilde v(x) \tilde v(y) - \bar v(x) \bar v(y)\|_2^2 \\
        &\le \int_{A_1 \times A_1} (\tilde v(x) \tilde v(y) - v(x) v(y))^2 \, dx\, dy
        + 2\int_{A_1 \times A_2} (\tilde v(x) \tilde v(y) - v(x) v(y))^2 \, dx\, dy \\
        &= \left(\int_{A_1} v^2(x)\, dx\right)^2 + 2 \int_{A_1} v^2(x)\, dx \int_{A_2} v^2(x) \, dx.
    \end{align*}
    Because $\bar v$ is non-negative,
    $|\tilde v(x) \tilde v(y) - \bar v(x) \bar v(y)| \le \Big||\tilde v(x) \tilde v(y)| - \bar v(x) \bar v(y)\Big|$,
    and then Lemma~\ref{lem:rounding-v-without-sign} implies that
    \begin{multline}\label{eq:rounding-including-cross-term}
        \|\tilde v(x) \tilde v(y) - v(x) v(y)\|_2^2
        \le
        C \int V_{2e-1}(\tilde v(x) \tilde v(y)) \, dx\, dy \\
        + \left(\int_{A_1} v^2(x)\, dx\right)^2
        + 2\int_{A_1} v^2(x)\, dx \int_{A_2} v^2(x)\, dx
    \end{multline}
    and it remains to bound the last two terms.

    Now, on $A_1 \times A_2$, we have $\tilde v(x) \tilde v(y) \le -{c^2}/{4}$, meaning that
    $V_{2e-1}(v(x) v(y)) = v^2(x) v^2(y)$.
    Hence,
    \begin{align*}
        \int_{[0,1]^2}  V_{2e-1}(v(x)v(y))\, dx\, dy
        &\ge \int_{A_1 \times A_2} v^2(x) v^2(y)\, dx\, dy \\
        &= \int_{A_1} v^2(x)\, dx \int_{A_2} v^2(y)\, dy.
    \end{align*}
    Moreover, we may assume that
    $\int_{A_1} v^2(x)\, dx \le \int_{A_2} v^2(y)\, dy$ (if not, this becomes true
    when we replace $v$ by $-v$).
    Then we can remove the last term of~\eqref{eq:rounding-including-cross-term}
    at the cost of increasing $C$ by 3.
\end{proof}

\begin{proof}[Proof of \Cref{prop:approx-B11}]
    Let $u$ be the function that we called $v$ in \Cref{lem:rounding-v}, and define $v = u - \int u\, dx$.
    Then $v$ trivially satisfies \cref{it:approx-B11-mean}.

    By \Cref{lem:rounding-v}, \Cref{lem:rank-1} and the triangle inequality,
    \[
        \|u(x) u(y) - \Delta g\|_2^2 \le C \int V_{2e-1}(\tilde v(x) \tilde v(y))\, dx\, dy + O(\delta^3) = O(\delta^3),
    \]
    with the second inequality coming from \Cref{cor:rank-1-values}.
    Since $\|\Delta g\|_2^2 = \delta^2 \pm O(\delta^3)$ by \Cref{lem:entropy-cost}, we have
    \[
        \|\Delta g\|_2 - \|u(x) u(y) - \Delta g\|_2 \le \|u(x) u(y)\|_2 = \|u\|_2^2 \le \|\Delta g\|_2 + \|u(x) u(y) - \Delta g\|_2,
    \]
    and so $\|u\|_2^2 = \delta \pm O(\delta^2)$. Since $u$ only takes 2 values (one of them 0, one of them bounded
    away from zero), we have $a := \int u \, dx = O(\delta)$. Hence $\|v\|_2^2 = \|u\|_2^2 - a^2 = \delta \pm O(\delta^2)$, proving \cref{it:approx-B11-norm}. The bound $a = O(\delta)$ also proves \cref{it:approx-B11:values}.

    For \cref{it:approx-B11-residual},
    \[
        \|v(x) v(y) - \Delta g\|_2 \le \|u(x) u(y) - \Delta g\|_2 + \|v(x) v(y) - u(x) u(y)\|_2;
    \]
    the first term on the right is $O(\delta^{3/2})$, while the second term is $\|a^2 - 2 a u\|_2 = O(\delta^{3/2})$ also.

    Having proven the first four claims, we will show that the last two claims
    follow by perturbing $C_2$ and $C_1$ slightly. That is, we are
    going to redefine $v$: it will still take two values, but we will change
    the sets on which it takes those two values, and then we will recenter it
    to maintain the property $\int v = 0$. If we only change the values of $v$
    on a set of size $O(\delta^2)$, the triangle inequality implies that items
    1--4 still hold for the modified function $v$.

    Let $A_2 \subset C_1$ consist of those $x$ for which
    \[
        \int_{C_1} (g(x,y) - (1-e))^2\, dy > \int_{C_1} (g(x,y) - e)^2\, dy,
    \]
    and let $A_1 \subset C_2$ consist of those $x$ for which
    \[
        \int_{C_1} (g(x,y) - (1-e))^2\, dy < \int_{C_1} (g(x,y) - e)^2\, dy.
    \]
    Note that for $x, y \in C_1$ we have $g(x,y) = (1-e) + \Theta(\delta) + r(x,y)$,
    and so if $x \in A_2$ then $\int_{C_1} r(x,y)^2 \, dy = \Omega(\delta)$. Similarly,
    if $x \in C_2$ and $y \in C_1$ then $g(x,y) = e + \Theta(\delta) + r(x,y)$ and so
    if $x \in A_1$ then $\int_{C_1} r(x,y)^2\, dy = \Omega(\delta)$. Since $\|r\|_2^2 = O(\delta^3)$,
    it follows that $|A_1 \cup A_2| = O(\delta^2)$.

    We now redefine $\Omega_i$ as follows: set $\tilde C_2 = C_2 \cup A_2 \setminus A_1$ and
    $\tilde C_1 = C_1 \cup A_1 \setminus A_2$. Since $g$ is bounded
    and $|\Omega_i \Delta \tilde \Omega_i| =
    O(\delta^2)$,~\eqref{eq:every-vertex-Omega1}
    and~\eqref{eq:every-vertex-Omega2} hold with $\tilde \Omega_i$ in place of
    $\Omega_i$. Redefining the function $v$ to be $\sqrt{2e-1} + O(\delta)$ on
    $\tilde C_1$ and $O(\delta)$ on $\tilde C_2$ (with the $O(\delta)$ terms chosen so that
    $\int v = 0$), \Cref{prop:approx-B11} holds with this modified function $v$.
\end{proof}

\section{Estimating the bipodal parameters}\label{sec:best-bipodal}

At this point, we have only shown that an entropy-optimal triangle-deficient graphon
is \emph{approximately} bipodal. For this section, we will temporarily switch to studying
truly bipodal graphons.
Up to measure-preserving transformations of $[0,1]$, every bipodal graphon
takes the form
$$
g(x,y) = \begin{cases} a & x,y < c, \cr 
b & x,y > c, \cr 
d & x < c < y \hbox{ or } y < c < x. \end{cases}
$$
That is, the pode sizes are $c$ and $1-c$, and without loss of 
generality we can assume that $c \le 1-c$.
We define 
$$
\Delta a = a-e, \qquad \Delta b = b-e, \qquad \Delta d = d-e,
\qquad \mu = \frac{c \Delta a}{1-c} + \Delta d.
$$

The main result of this section is that in the class of bipodal graphons
with edge density $e$, triangle density $\tau = e^3 - \delta^3$,
and parameters $\Delta b = o(1)$, $c = o(1)$, $\Delta d = o(1)$, $\mu = o(\delta)$, and $|\Delta a| = \Omega(1)$,
there is a unique entropy-optimal bipodal graphon and we get good estimates on its parameters.

To be precise, Let $\calG_{e,\delta,\eta}$ be the set of bipodal graphons with edge-density $e$,
triangle density $e^3 - \delta^3$, and parameters
$a, b, c$ and $d$ satisfying $|b - e| < \eta \sqrt{\delta}$, $c < \eta$, $|d - e| < \eta$, $|a - e| > \eta$, and $|\mu| = O(\delta^{3/2})$.

\begin{proposition}\label{prop:bipodal-uniqueness-and-estimates}
    For every $\frac 12 < e < 1$ there exists $\eta > 0$ such that for any $\delta < \eta$ there is at most
    one graphon $g \in \calG_{e,\delta,\eta}$ maximizing $S(g)$. Moreover,
    any such optimal graphon has parameters satisfying
    \begin{eqnarray}
a & = & 1-e - \delta + O(\delta^2) \cr 
b & = & e - \frac{\delta^2}{2e-1} + O(\delta^3) \cr 
c & = & \frac{\delta}{2e-1} - \frac{2\delta^2}{2e-1} + O(\delta^3) \cr 
d & = & e + \delta + \frac{\delta^2}{eH'(e)} \left ( H'(e) - \left (e-\frac12 \right ) H''(e) \right ) + O(\delta^3). \end{eqnarray}
\end{proposition}

To prove \Cref{prop:bipodal-uniqueness-and-estimates}, we first show that the parameters of any optimal graphon
must satisfy the claimed estimates, so assume that $g$ is optimal.
We define 
$$
\Delta a = a-e, \qquad \Delta b = b-e, \qquad \Delta d = d-e,
\qquad \mu = \frac{c \Delta a}{1-c} + \Delta d.
$$
The edge density, triangle density and entropy of the graphon are:
\begin{eqnarray}
\edens(g) &=&  c^2 a + 2c(1-c) d + (1-c)^2 b, \label{e-formula} \\
\tdens(g) & = & c^3 a^3 + 3c^2(1-c)ad^2 + 3c(1-c)^2 bd^2 + (1-c)^3 b^3, 
\label{t-formula} \\
S(g) & = & c^2 H(a) + 2c(1-c)H(d) + (1-c)^2 H(b). \label{S-formula}
\end{eqnarray} 

\subsection{Expressing all quantities in terms of $a$ and $\mu$}
Our constraint on edge count is then
$$
0 = \edens - e = c^2 \Delta a + 2c(1-c) \Delta d + (1-c)^2 \Delta b.
$$ 
Combined with the definition of $\mu$, this gives 
\begin{equation}\label{b-sub}
\Delta b 
= \frac{c}{1-c} \left ( \frac {c \Delta a}{1-c} - 2\mu \right ), 
\qquad \Delta d = \mu - \frac{c \Delta a}{1-c}.
\end{equation}

We then turn to the triangle count. Plugging $a = e + \Delta a$,
$b = e + \Delta b$, and $d = e + \Delta d$ into equation 
(\ref{t-formula}) gives
\begin{eqnarray}\label{delta-tau}
\tau - e^3 & = & 3e^2 (c^2 \Delta a + 2c(1-c) \Delta d + (1-c)^2 \Delta b) \cr 
&& + 3e \left ( c(c \Delta a + (1-c)\Delta d)^2 + (1-c)
(c \Delta d + (1-c) \Delta b)^2 \right ) \cr 
&& + c^3 \Delta a^3 + 3c^2(1-c) \Delta a \Delta d^2 + 3c^2(1-c)^2 
\Delta b \Delta d^2 + (1-c)^3 \Delta b^3. 
\end{eqnarray}
The first line is $3e^2(\varepsilon - e) = 0$. The second line
works out to 
$$
3e [c((1-c)\mu)^2 + (1-c)(-c\mu)^2 ] = 3ec(1-c) \mu^2.
$$
The terms in the last line are 
$$(1-c)^3 \left (\frac{c \Delta a}{1-c} \right )^3,
$$
$$ 
3c(1-c)^2 \left [ \left ( \frac{c \Delta a}{1-c}\right )^3 - 2\mu
\left ( \frac{c \Delta a}{1-c}\right )^2 + \mu^2 \left ( \frac{c \Delta a}{1-c}\right ) \right ],
$$
$$
3c^2(1-c) \left [ \left ( \frac{c \Delta a}{1-c}\right )^3 - 4\mu
\left ( \frac{c \Delta a}{1-c}\right )^2 + 5\mu^2 \left ( \frac{c \Delta a}{1-c}\right ) -2\mu^3 \right ],
$$
and 
$$
c^3 \left [ \left ( \frac{c \Delta a}{1-c}\right )^3 - 6\mu
\left ( \frac{c \Delta a}{1-c}\right )^2 + 12\mu^2 \left ( \frac{c \Delta a}{1-c}\right ) -8\mu^3 \right ].
$$
Setting the sum of the second line of (\ref{delta-tau}) 
and these four terms equal to $-\delta^3$ gives 
\begin{eqnarray}\label{delta-sum}
- \delta^3 & = & \left ( \frac{c \Delta a}{1-c}\right )^3
- 6 c \mu \left ( \frac{c \Delta a}{1-c}\right )^2 \cr 
&& + 3c\mu^2 \left ( (1+3c) \left ( \frac{c \Delta a}{1-c}\right )
+ e(1-c) \right ) -(6c^2+2c^3) \mu^3.
\end{eqnarray}

The definition of $\calG_{e,\delta,\eta}$ implies that $\mu = O(\delta^{3/2})$.
This means that the terms in (\ref{delta-sum}) involving $\mu$ do not 
affect $\tdens-e^3$ to leading order, so we have  
$$
 \left ( \frac{c \Delta a}{1-c}\right )^3 = -\delta^3 + o(\delta^3).
$$
 In particular, $c$ is of order $\delta$. This in turn implies that the 
 $o(\delta^3)$ terms in (\ref{delta-sum}), which go as 
 $c^3\mu$ and $c\mu^2$ and higher
 powers of $c$ and $\mu$, are actually $O(\delta^{9/2})$, so 
 \begin{eqnarray*}
 \frac{c \Delta a}{1-c} &=& -\delta + O(\delta^{5/2}), \cr 
 c & = & \frac{-\delta}{\Delta a} + O(\delta^2). 
 \end{eqnarray*}
 The leading corrections to the approximation ${c \Delta a}/({1-c})
 \approx - \delta$ come from the terms
\vskip 0 truein \noindent
 $-6c \mu \left ( c\Delta a/({1-c})\right )^2$ and 
$3ec(1-c)\mu^2$ in the expansion of $\tau - e^3$. A priori we 
don't know which is larger, so for now we will keep both. (They will both 
turn out to be of order $\delta^5$). However,
all other terms are at least one power of $\delta$ smaller
than one (or both) or these terms. We can use the approximation
$c \approx -\delta/\Delta a$ to simplify these higher-order corrections:
\begin{eqnarray*}
\left ( \frac{c \Delta a}{1-c} \right )^3 &= & 
-\delta^3 +\frac{3e \mu^2 \delta -6 \mu \delta^3}{\Delta a} + 
O(\mu\delta^4, \mu^2\delta^2), \cr 
\frac{c \Delta a}{1-c} & = &  -\delta +  \frac{e \mu^2 -2\mu\delta^2}{\delta \Delta a} +O(\mu\delta^2, \mu^2),
\end{eqnarray*}
where $O(\mu\delta^4,\mu^2\delta^2)$ is shorthand for 
$O(\mu\delta^4)+O(\mu^2\delta^2)$. 
From that we compute $c$:
\begin{eqnarray} \label{c-estimate}
c & = & \frac{- \delta + \frac{e \mu^2 -2\mu\delta^2}{\delta \Delta a}}{ \Delta a - \delta + \frac{e \mu^2 -2\mu\delta^2}{\delta \Delta a}} + O(\mu\delta^2, \mu^2) \cr 
&=& \frac{\delta + \frac{2\mu\delta^2 - e\mu^2}{\delta \Delta a}}{ \delta - \Delta a} + O(\mu\delta^2, \mu^2). 
\end{eqnarray}
This then determines $\Delta b$ and $\Delta d$:  
\begin{eqnarray*} 
\Delta b &=& \frac{c}{1-c} \left ( \frac{c\Delta a}{1-c}-2\mu
\right ) \cr 
& = &  \frac{\delta^2}{\Delta a} + \frac{2 \mu \delta}{\Delta a}
+ \frac{4\mu \delta^2 - 2e\mu^2}{\Delta a}+ O(\mu\delta^3,\mu^2\delta), \cr 
\Delta d & = & \mu + \delta + \frac{2 \mu \delta^2-\mu^2e}{\delta \Delta a} + O(\mu \delta^2, \mu^2).
\end{eqnarray*}

Note that the constraints on $\edens$ and $\tdens$ are algebraic, so 
we could have expressed $b$, $c$ and $d$ as power series in $a$ and $\mu$. 
However, in a power series the derivative of a term with respect 
to $\mu$ is at most an order 
$\mu^{-1}$ larger than the term itself, while the derivative with respect to $a$ is at
most of the same order as the term. We can therefore turn our estimates of $(b,c,d)$ into
estimates of $(\partial_a b, \partial_a c, \partial_a d)$ and $(\partial_\mu b, \partial_\mu 
c, \partial_\mu d)$. Specifically: 
\begin{eqnarray*}
\partial_a b & = & \frac{-\delta^2}{\Delta a^2} - \frac{2\mu\delta}{\Delta a^2} + \frac{4 \mu^2 e - 8\mu \delta^2}{\Delta a^3} + O(\mu^2\delta, \mu\delta^3) \cr 
& = & -\frac{c^2}{(1-c)^2} -\frac{2 \mu \delta}{\Delta a^2} + \frac{2 \mu^2e -8\mu\delta^2}{\Delta a^3} + O(\mu^2\delta, \mu \delta^3), \cr 
\partial_a c & = & \frac{c}{\delta - \Delta a} + O(\mu\delta, \mu^2/\delta), \cr 
\partial_a d &=& \frac{e\mu^2 - 2\mu \delta^2}{\delta \Delta a^2} +O(\mu\delta^2,\mu^2).
\end{eqnarray*}
\begin{eqnarray*}
\partial_\mu b & = & \frac{2\delta}{\Delta a} + \frac{4\delta^2 - 4e\mu}{\Delta a^2} + O(\mu\delta, \delta^3), \cr 
\partial_\mu c & = & \frac{2 \mu e - 2 \delta^2}{\delta \Delta a^2} + O(\mu, \delta^2), \cr 
\partial_a d &=& 1 + \frac{2 \delta^2-2e\mu}{\delta \Delta a} +O(\delta^2,\mu).
\end{eqnarray*}
 
 \subsection{Solving $\partial_a S = \partial_\mu S = 0$}
We now solve the equations $\partial_a S = \partial_\mu S = 0$ in three passes. First we
solve $\partial_a S = 0$ to lowest order, obtaining $a$ to within $O(\delta)$. Using this value of $a$, we solve
$\partial_\mu S = 0$, showing that $\mu$ is a specific constant times $\delta^2$, up to 
an $O(\delta^3)$ error. Finally, we solve $\partial_a S = 0$ more precisely, determining 
$a$ to order $\delta$, with an $O(\delta^2)$ error. This then determines $(a,b,c,d)$ to 
the accuracy specified in \Cref{prop:bipodal-uniqueness-and-estimates}.

Since
\begin{equation}\label{WhatsS}
S = c^2 H(a) + 2c(1-c) H(d) + (1-c)^2 H(b), 
\end{equation}
\begin{eqnarray*}
\partial_a S &=& c^2 H'(a) + 2c(1-c)H'(d)\partial_a d+ (1-c)^2H'(b) \partial_a b \cr 
&& + 2 \partial_a c (H(d)-H(b) + c(H(a)+ H(b) - 2H(d))). 
\end{eqnarray*}
Keeping terms through $O(\delta^2)$, and noting that all discarded terms are of order
$\delta^3$ or higher, we have 
\begin{eqnarray*}
\partial_a S & = & c^2 H'(a) - c^2 H'(b) - \frac{2c}{\Delta a} \left (\delta H'(e) 
+ c(H(a)-H(e)\right ) + O(\delta^3) \cr 
& = & c^2 (H'(a)+H'(e)) -2c^2 \frac{H(a)-H(e)}{a-e} + O(\delta^3).
\end{eqnarray*}
Setting this equal to zero and dividing by $\delta^2$ gives 
$$ 
H'(a) + H'(e) - 2 \frac{H(a)-H(e)}{a-e} = O(\delta). 
$$
This implies that $a$ is within $O(\delta)$ of either $e$ or $1-e$. Since $a-e$ is assumed
not to be $o(1)$, this means that we can write 
$$ 
a = 1-e + a_1,
$$
where $a_1$ is $O(\delta)$. This allows us to expand $H(a)$ and $H'(a)$ a power series
in $a_1$ with quantified errors and also to quantify how much 
${-1}/{\Delta a}$ and ${1}/({\delta - \Delta a})$ differ from ${1}/({2e-1})$. 
In particular, 
$$
H(d)-H(b) + c(H(a)+H(b)-2H(d))  =  
H'(e)\delta + O(\mu, \delta^2).
$$

We now evaluate 
\begin{eqnarray*}
\partial_\mu S &=& 2c(1-c)H'(d)\partial_\mu d+ (1-c)^2H'(b) \partial_\mu b \cr 
&& + 2 \partial_\mu c (H(d)-H(b) + c(H(a)+ H(b) - 2H(d))) 
\end{eqnarray*}
through order $\mu, \delta^2$: 
\begin{eqnarray*} 
2c(1\!-\!c)H'(d) \partial_\mu d & = & 
2 \left ( \frac{\delta}{2e\!-\!1} + \frac{a_1 \delta - 2\delta^2}{(2e\!-\!1)^2}\right ) 
\Big (H'(e) \!+\! \delta H''(e) \Big ) \left (1 \!+\! \frac{2 e \mu \!-\! 2\delta^2}{\delta(2e\!-\!1)}
\right ) \!+\! o(\mu, \delta^2) \cr 
& = & H'(e) \left ( \frac{2\delta}{2e-1} + \frac{2a_1 \delta - 8\delta^2 + 4e \mu}{(2e-1)^2}
\right ) + \frac{2 H''(e) \delta^2}{2e-1} + o(\mu, \delta^2),  \cr 
(1\!-\!c)^2 H'(b) \partial_\mu b & = & H'(e) \left (1 - \frac{2\delta}{2e-1}\right ) 
\left ( \frac{-2\delta}{2e-1} + \frac{4\delta^2 - 4e\mu -2a_1\delta}{(2e-1)^2} \right )
+ o(\mu, \delta^2) \cr 
& = & H'(e) \left (\frac{-2\delta}{2e-1} + \frac{8\delta^2-4e\mu -2a_1 \delta}{(2e-1)^2}
\right ) +o(\mu, \delta^2). \end{eqnarray*} 
$$  
2 \partial_\mu c (H(d)-H(b)+c(H(a)+H(b)-2H(d)))  =  
2 \left ( \frac{2 e \mu - 2 \delta^2}{\delta (2e-1)^2} \right ) H'(e) \delta 
+o(\mu, \delta^2).
$$
Adding up the terms in $\partial_\mu S$ and setting the total equal
to zero, we have 
\begin{eqnarray}\label{WhatsMu}
\frac{4H'(e)(\mu e - \delta^2)}{(2e-1)^2} + \frac{2 H''(e) \delta^2}{2e-1} & = & o(\mu,\delta^2), \cr 
e \mu - \delta^2 & = & - \frac{H''(e) \delta^2(2e-1)}{2H'(e)} + o(\mu,\delta^2), \cr 
\mu & = & \frac{\delta^2}{e} \left ( 1 - \frac{H''(e)(2e  - 1)}{2H'(e)}\right ) + o(\delta^2,\mu).
\end{eqnarray}
Now that we have established that $\mu=O(\delta^2)$, we can check the 
order of the error terms in our estimate of $\partial_\mu S$. They are all 
$O(\delta^3)$, not just $o(\mu, \delta^2)$, so we have actually estimated $\mu$ to 
within $O(\delta^3)$.

Using our known value of $\mu$, we can restate our estimates for the derivatives of 
$(b,c,d)$ as 
\begin{eqnarray*}
\partial_a b & = & \frac{-\delta^2}{\Delta a^2} - \frac{2 \mu\delta}{\Delta a^2} +O(\delta^4) \cr 
& = & \frac{-c^2}{(1-c)^2} - \frac{2c\mu}{2e-1} + O(\delta^4), \cr 
\partial_a c &=& \frac{c}{\delta - \Delta a} + O(\delta^3) \cr
& = & \frac{c^2}{\delta} + O(\delta^3), \cr 
\partial_ad & = & O(\delta^3). 
\end{eqnarray*}

We then compute the first three terms in $\partial_aS:$
\begin{eqnarray*}
c^2 H'(a) & = & c^2(-H'(e) + a_1 H''(e) ) + O(\delta^4), \cr 
(1-c)^2  \partial_a b H'(b) & = & -c^2 H'(e) - \frac{2 c \mu}{2e-1} H'(e) + O(\delta^4), \cr 
2c(1-c)\partial_a d H'(d) & = & O(\delta^4). 
\end{eqnarray*}
The last term, namely $2 \partial_a c (H(d)-H(b) + c(H(a)+H(b)-2H(d)))$, works out to 
\begin{eqnarray*}
&&\frac{2c^2}{\delta}\left ( \frac{H''(e)\delta^2}{2} + H'(e) \left ( \delta + \mu + \frac{\delta^2}{2e-1} + c(-a_1 - 2\delta) \right ) \right ) + O(\delta^4) \cr 
&=& \frac{2c^2}{\delta} \left ( \frac{H''(e)\delta^2}{2} + H'(e) (\delta + \mu - c(a_1+\delta) \right ) + O(\delta^4).
\end{eqnarray*} 

Adding everything together, we have 
$$ 
\partial_a S = c^2(a_1+\delta) \left ( H''(e) - \frac{2 H'(e)}{2e-1} \right ) + O(\delta^4),
$$
so 
$$
a_1 = -\delta + O(\delta^2). 
$$
Applying this value, and the computed value of $\mu$, to our expressions for $(a,b,c,d)$,
we obtain the estimates of \Cref{prop:bipodal-uniqueness-and-estimates}.

\subsection{Uniqueness}\label{sec:unique}

Having proven that any optimizing bipodal graphon must have certain parameter estimates,
we now show that it is unique. We compute the Hessian of $S$ with respect to 
$a$ and $\mu$ in a neighborhood of the optimal graphon, specifically in the neighborhood
defined by the estimates in \Cref{prop:bipodal-uniqueness-and-estimates},
and show that is is well-approximated by a fixed non-degenerate matrix. That is, the 
equations $\partial_a S = \partial_\mu S = 0$ are approximately linear and non-degenerate 
in this region, and so have a unique solution.

The partial derivatives of $b$, $c$, and $d$ are, to leading order, 
\begin{eqnarray*}
\partial_a b = \frac{-\delta^2}{(2e-1)^2}+O(\delta^3), & 
\partial_\mu b = \frac{-2\delta}{2e-1} + O(\delta^2), & & \cr 
\partial^2_{aa} b = \frac{-2\delta^2}{(2e-1)^3} +O(\delta^3), &  
\partial^2_{\mu\mu}b = \frac{-4e}{(2e-1)^2} + O(\delta), & 
\partial^2_{a\mu}b = \frac{-2\delta}{(2e-1)^2} + O(\delta^2).
\end{eqnarray*}
\begin{eqnarray*}
\partial_a c = \frac{\delta}{(2e-1)^2}+O(\delta^2), & 
\partial_\mu c = \frac{2e\mu-2\delta^2}{\delta(2e-1)^2} + O(\delta^2), & \cr 
\partial^2_{aa} c = \frac{2\delta}{(2e-1)^3} +O(\delta^2), &  
\partial^2_{\mu\mu}c = \frac{2e}{\delta(2e-1)} + O(1), & 
\partial^2_{a\mu}c = \frac{4e\mu-4\delta^2}{\delta(2e-1)} + O(\delta^2).
\end{eqnarray*}
\begin{eqnarray*}
\partial_a d = \frac{e\mu^2-2\mu\delta^2}{\delta(2e-1)^2}+O(\delta^4), & 
\partial_\mu d = 1 + O(\delta), & \cr 
\partial^2_{aa} d = \frac{2e\mu^2 - 4\mu\delta^2}{\delta (2e-1)^3} +O(\delta^4), &  
\partial^2_{\mu\mu}d = \frac{2e}{\delta(2e-1)} + O(1) & 
\partial^2_{a\mu}d = \frac{2e\mu-2\delta^2}{\delta(2e-1)^2} + O(\delta^2).
\end{eqnarray*}

We compute $\partial^2_{aa} S$ to order $\delta^2$. Four terms contribute to that 
order, namely 
\begin{eqnarray*}
2 \partial^2_{aa}c (H(d)-H(b)) \approx \frac{4\delta^2 H'(e)}{(2e-1)^3}, && 
4c \partial_a c H'(a) \approx \frac{-4\delta^2 H'(e)}{(2e-1)^3}, \cr 
\partial^2_{aa}b H'(b) \approx \frac{-2\delta^2 H'(e)}{(2e-1)^3}, && 
c^2 H''(a) \approx \frac{\delta^2 H''(e)}{(2e-1)^2}.
\end{eqnarray*}
Adding up these terms gives 
$$ 
\partial^2_{aa} S = \frac{\delta^2}{(2e-1)^2} \left ( H''(e) - \frac{2 H'(e)}{2e-1}\right )
+ O(\delta^3). 
$$
In the expansion of $\partial^2_{\mu\mu} S$, the unique $O(\delta^{-1})$ term is 
$2(1-2c)\partial^2_{\mu\mu} d H'(d)$, giving us 
$$ \partial^2_{\mu\mu} S= \frac{4eH'(e)}{\delta(2e-1)} + O(1). 
$$ 
Finally, in computing $\partial^2_{a\mu}S$ there are two $O(\delta)$ terms, namely 
$(1-c)^2 \partial^2_{a\mu} b H'(b)$ and 
\newline $2(1-2c) \partial_a c \partial_\mu d H'(d)$.
The first gives ${-2\delta H'(e)}/{(2e-1)^2} + O(\delta^2)$ while the second gives 
${+2\delta H'(e)}/{(2e-1)^2} + O(\delta^2)$, so 
$$ 
\partial^2_{a\mu} S = O(\delta^2).
$$

Since the Hessian of $S$ takes the form 
$$ \begin{pmatrix} K_1 \delta^2+O(\delta^3) & O(\delta^2) 
\cr O(\delta^2) & K_2 \delta^{-1} + O(1) \end{pmatrix}, $$
for negative constants $K_1$ and $K_2$, the entropy has a unique maximizer, and indeed 
a unique critical point, in the region governed by the 
estimates of \Cref{prop:bipodal-uniqueness-and-estimates}.
At first glance, the off-diagonal terms appear as big as the $\partial^2_{aa}S$ term. However, replacing the variable $\mu$ with $\mu\delta^{-3/2}$ would convert the Hessian to the form 
$$\begin{pmatrix} K_1 \delta^2+O(\delta^3) & O(\delta^{7/2}) 
  \cr O(\delta^{7/2}) & K_2 \delta^{2} + O(\delta^3) \end{pmatrix},$$
in which the 
known diagonal terms more manifestly dominate the error terms.

This completes the proof of \Cref{prop:bipodal-uniqueness-and-estimates}.

\section{Averaging}\label{sec:averaging}

We now return to the study of entropy-optimal graphons that are not necessarily bipodal:
let $g$ be a graphon maximizing $S(g)$ subject to $\edens(g) = e$, $\tdens(g) = e^3 - \delta^3$,
and recall from \Cref{prop:approx-B11} that we can partition $[0, 1]$ into podes $C_1$ and $C_2$
so that $g$ is approximately bipodal with respect to these podes.
Recall also the definition of $v$ from \Cref{prop:approx-B11}, and that $g(x, y) \approx e - v(x) v(y)$.
Let $h$ be the graphon obtained by averaging $g$ on the podes $C_i \times C_j$. We will write $h = g + \Delta h$.

In this section we show that $g$ must be \emph{exactly} bipodal. Specifically, we show that if $g
\ne h$ then we can get a contradiction by constructing a $\tilde h$ with
$\tdens(\tilde h) = \tdens(g)$ but $S(\tilde h) > S(g)$.
This argument comes in four parts:
\begin{enumerate}
    \item we give a lower bound on $S(h) - S(g)$, and we show that if the bound
        is almost sharp then $\Delta h$ has a particular structure;
    \item we give an upper bound on $\tdens(h) - \tdens(g)$, and we show that if the bound
        is almost sharp then $\Delta h$ has a particular structure;
    \item we find a perturbation $\tilde h$ of $h$ that trades off entropy for triangles
        at essentially the best possible ratio of the preceding two bounds, and it follows
        that if $g$ is optimal then both of the preceding two bounds must be almost sharp; and finally,
    \item we show that the structure of $\Delta h$ implied by the first two parts
        is incompatible with \Cref{prop:approx-B11}.
\end{enumerate}

One basic observation before we start is that because $v(x) v(y)$ is constant on $\Omega_i \times \Omega_j$,
and because the mean of any function is the constant with the smallest $L^2$ distance to that function,
\begin{equation}\label{eq:Delta-h-size}
    \|\Delta h\|_2^2 \le \|g - v(x) v(y)\|_2^2 = O(\delta^3),
\end{equation}
where the last inequality follows from \Cref{prop:approx-B11}.

\subsection{The entropy change}

Let's compute the change in entropy that results from replacing each pode in $g$ by its average in $h$.
Recall that $C_2 \subset [0, 1]$ is the set on which $v(x) \approx 0$ and $C_1 = [0, 1]\setminus C_2$
is the set on which $v(x) \approx \sqrt{2e-1}$.
By \Cref{prop:approx-B11} and the fact that $|C_1| = \Theta(\delta)$,
\[
    h(x,y) = \begin{cases}
        h_{11} = 1-e + O(\delta^{1/2}) & \text{ on $C_1 \times C_1$} \\
        h_{12} = e + O(\delta) & \text{ on $C_2 \times C_1$ and $C_1 \times C_2$} \\
        h_{22} = e + O(\delta^{3/2}) & \text{ on $C_2 \times C_2$}
    \end{cases}
\]
In terms of the notation of \Cref{sec:best-bipodal}, $h_{11} = a$, $h_{12} = d$, $h_{22} = b$, and
$|C_1| = c$.

Because $h$ is obtained by averaging $g$, it has larger entropy. Our first bound shows that
it must be larger by at least about the ``optimal entropy-$L^2$ tradeoff constant,'' $C(e)$.

\begin{lemma}\label{lem:avg-entropy-change}
\[
    S(g) \le S(h) - C(e) (1 - O(\sqrt \delta)) \|g - h\|_2^2.
\]
\end{lemma}

\begin{proof}
Let $D(p, q) = p \ln \frac pq + (1-p) \ln \frac{1-p}{1-q}$,
and note that for any set $A \subset [0, 1]^2$,
\[
    \int_A D(g(x,y), q) \, dx\, dy = H(q) |A| - \int_A H(g(x,y))\, dx\, dy.
\]
In particular, if we apply this to $A =  \Omega_i \times \Omega_j$ then
\[
    \int_{\Omega_i \times \Omega_j} H(h(x,y)) - H(g(x,y))\, dx\, dy = \int_A D(g(x,y), h_{ij})\, dx\, dy.
\]
Recalling the definition of $C(e)$ from~\eqref{eq:C-def},
we have
\begin{equation}\label{eq:avg-entropy-change-ineq}
    \int_A D(g(x,y), h_{ij}) \, dx\, dy \ge C(h_{ij}) \int_A (g(x,y) - h_{ij})^2\, dx\, dy.
\end{equation}
Since $C(e) = C(1-e)$ and $C$ is continuous and differentiable at $e$,
$C(h_{ij}) = C(e) + O(\delta^{1/2})$ for every $i$ and $j$. Hence,
\[
    \int_A D(g(x,y), h_{ij}) \, dx\, dy \ge C(e)(1 - O(\sqrt \delta)) \int_A (g(x,y) - h(x,y))^2\, dx\, dy,
\]
and summing over $i$ and $j$ gives
\[
    S(g) \le S(h) - C(e) (1 - O(\sqrt \delta)) \|g - h\|_2^2.
    \qedhere
\]
\end{proof}

Next, we show that unless $g$ approximately takes the value $e$ and $1-e$, the
entropy gain of $h$ is even larger than in \Cref{lem:avg-entropy-change}. The
argument here is similar to that of~\Cref{lem:entropy-cost}; the difference is
that the cost here is measured in terms of $\|\Delta h\|_2^2$, instead of in terms
of $\|\Delta g\|_2^2$ (which is much larger).

\begin{lemma}\label{lem:avg-entropy-change-better-2}
    Fix $\eta > 0$. For sufficiently small $\delta$ (depending on $\eta$), if
    \[
        \int_{C_1 \times C_1} V_{2e-1}(-\Delta h)\, dx\, dy \ge \eta \|\Delta h\|_2^2
    \]
    or
    \[
        \int_{(C_1 \times C_1)^c} V_{2e-1}(\Delta h)\, dx\, dy \ge \eta \|\Delta h\|_2^2
    \]
    then
    \[
        S(g) \le S(h) - C(e)(1 + \Omega(\eta)) \|\Delta h\|_2^2.
    \]
\end{lemma}

\begin{proof}
    We follow the same argument as in \Cref{lem:avg-entropy-change}, but in~\eqref{eq:avg-entropy-change-ineq}
    we use the improved bound of \Cref{lem:C-prop} to obtain
    \[
        \int_{\Omega_i \times \Omega_j} D(g(x,y), h_{ij}) \, dx\, dy \ge \int_{\Omega_i \times \Omega_j} (C(h_{ij}) + \Omega((g - 1 + h_{ij})^2)(g - h_{ij})^2\, dx\, dy.
    \]
    In particular, compared to~\eqref{eq:avg-entropy-change-ineq}, the right hand side is increased by
    \[
        \Omega\left(
            \int_{\Omega_i \times \Omega_j} \min\{|g - 1 + h_{ij}|, |g - h_{ij}|\}^2\, dx\, dy
        \right)
        = \Omega\left(\int_{\Omega_i \times \Omega_j} V_{1 - 2 h_{ij}} (\Delta h)\, dx\, dy\right).
    \]
    Recalling that $h_{22} = 1 - e + o(1)$, and $h_{ij} = e + o(1)$ otherwise, \Cref{lem:V-value-change}
    implies that
    \[
        \int_{C_1 \times C_1} D(g(x,y), h_{22}) \, dx\, dy \ge C(h_{22}) |C_1|^2
        + \int_{C_1 \times C_1} V_{2e-1} (-\Delta h)\, dx\, dy - o(\|\Delta h\|_2^2),
    \]
    and
    \[
        \int_{\Omega_i \times \Omega_j} D(g(x,y), h_{ij}) \, dx\, dy \ge C(h_{ij}) |\Omega_i| |\Omega_j|
        + \int_{\Omega_i \times \Omega_j} V_{2e-1} (\Delta h)\, dx\, dy - o(\|\Delta h\|_2^2),
    \]
    for $(i, j) \ne (2, 2)$.
    Now we simply sum over $i,j$ as in the proof of \Cref{lem:avg-entropy-change}: under our assumptions,
    at least one of the $\int_{\Omega_i \times \Omega_j} V_{2e-1} (\pm \Delta h)\, dx\, dy$ terms
    gives an additional contribution of $\Omega(\eta \|\Delta h\|_2^2)$ compared to \Cref{lem:avg-entropy-change}.
\end{proof}

\subsection{The triangle change}

In this section we will control the triangle change between $g$ and $h$.
Recall that $T_{\Delta h}: L^2([0, 1]) \to L^2([0, 1])$ denotes the integral operator with kernel $\Delta h$.
Recall that
\begin{align}
    \tdens(h) - \tdens(g)
    &= 3 \int h(x,y) \Delta h(y,z) \Delta h(z,x) + 3 \int h(x,y) h(y,z) \Delta h(z,x) \notag \\ 
    & \phantom{=} + \int \Delta h(x,y) \Delta h(y,z) \Delta h(z,x) \notag \\
    &= 3 \int h(x,y) \Delta h(y,z) \Delta h(z,x) + O(\|\Delta h\|_2^3) \notag \\
    &\ge - 3 \|T_{\Delta h} v\|_2^2 + O(\delta^{3/2} \|\Delta h\|_2^2),
    \label{eq:triangle-expansion}
\end{align}
where the last line follows because (by \Cref{prop:approx-B11})
$h = e - v(x) v(y) - \tilde r(x,y)$ for some $\|\tilde r\|_2 \le \delta^{3/2}$,
and the term $\int e \Delta h(y,z) \Delta h(z,x)\, dx\, dy\, dz$ is non-negative.

\begin{lemma}\label{lem:orthogonality}
    \[
        \inr{v}{T_{\Delta h}v} = 0
    \]
\end{lemma}

\begin{proof}
    We can write
    \[
        \inr{v}{T_{\Delta h}v} = \int \Delta h(x,y) v(x) v(y)\, dx\, dy;
    \]
    recall that $v(x) v(y)$ is constant on every pode and $\Delta h(x,y)$ integrates to zero on each pode.
\end{proof}

Now define $u = \|v\|_2^{-2} T_{\Delta h} v$, define $w(x,y)$ by $\Delta h(x,y) = u(x) v(y) + u(y) v(x) + w(x,y)$,
and write $T_w$ for the integral operator with kernel $w$. By \Cref{lem:orthogonality},
$\inr{u}{v}= 0$. Then $u = T_{\Delta h}v = u + T_w v$, and so $T_w v = 0$. It follows that
\[
    \|T_{\Delta h}\|^2_2 = \|\Delta h\|_2^2 = 2 \|u\|_2^2 \|v\|_2^2 + \|w\|_2^2 = 2 \frac{\|T_{\Delta h} v\|_2^2}{\|v\|_2^2} + \|w\|_2^2.
\]
Since $\|w\|_2^2 \ge 0$, $\|T_{\Delta h} v\|_2^2 \le \frac 12 \|\Delta h\|_2^2 \|v\|_2^2$,
and then~\eqref{eq:triangle-expansion}
gives the following bound:
\begin{lemma}\label{lem:avg-triangle-change}
\[
    \tdens(h) \le \tdens(g) + \frac 32 \|\Delta h\|_2^2 \|v\|_2^2 - 3 \|w\|_2^2 \|v\|_2^2 + O(\delta^{3/2} \|\Delta h\|_2^2).
\]
\end{lemma}
We should interpret \Cref{lem:avg-triangle-change} as saying that $\tdens(h)
\lesssim \tdens(g) + \frac 32 \|\Delta h\|_2^2 \|v\|_2^2$, with approximate
tightness only if $\|w\|_2^2$ is small compared to $\|\Delta h\|_2^2$.

Let's also note that $\|u\|_2$ must be small:
\begin{lemma}\label{lem:u-small}
    \[
        \|u\|_2^2 = O(\delta^2).
    \]
\end{lemma}

\begin{proof}
    \[
        2 \|u\|_2^2 \|v\|_2^2 \le \|\Delta h\|_2^2 = O(\delta^3),
    \]
    where the second inequality follows from~\eqref{eq:Delta-h-size}.
    Finally, $\|v\|_2^2 = \Theta(\delta)$.
\end{proof}

In order for~\Cref{lem:avg-triangle-change} to be sharp, $\|w\|_2$
must be small compared to $\|\Delta h\|_2$. On the other hand,
\Cref{lem:avg-entropy-change-better-2} implies that if the entropy change
inequality of~\Cref{lem:avg-entropy-change} is sharp then $\Delta h$ must spend
most of its $L^2$ mass at particular values. Combining these two pieces of
tells us that $u$ must spend most of its $L^2$ mass at particular values.

\begin{lemma}\label{lem:u-structure}
    Fix $\eta \ge \delta > 0$ and suppose that $\|w\|_2^2 \le \eta \|\Delta h\|_2^2$.

    If $\int_{C_1 \times C_1} V_{2e-1}(-\Delta h)\, dx\, dy \le \eta \|\Delta h\|_2^2$
    then
    \[\int_{C_1} V_{\sqrt{2e-1}}(-u) \le O(\eta) \|u\|_2^2.\]

    If $\int_{(C_1 \times C_1)^c} V_{2e-1}(\Delta h)\, dx\, dy \le \eta \|\Delta h\|_2^2$
    then
    \[\int_{C_2} V_{\sqrt {2e-1}}(u) \le O(\eta) \|u\|_2^2.\]
\end{lemma}

\begin{proof}
    Let $v_2 = \sqrt{2e-1} \pm O(\delta)$ be the value that $v$ takes on $C_1$.
    On $C_1 \times C_1$, we have $\Delta h = v_2 (u(x) + u(y)) + w(x,y)$.
    If $\int_{C_1 \times C_1} V_{2e-1}(-\Delta h)\, dx\, dy \le \eta \|\Delta h\|_2^2$
    then (by \Cref{lem:V-perturbation})
    \[
        \int_{C_1 \times C_1} V_{2e-1}(-v_2 (u(x) + u(y)))\, dx\, dy \le O(\eta) \|\Delta h\|_2^2 + O(\|w\|_2^2) = O(\eta) \|\Delta h\|_2^2.
    \]
    Since $(2e-1) / v_2 = \sqrt{2e-1} + O(\delta)$, \Cref{lem:V-value-change} implies that
    \[
        \int_{C_1 \times C_1} V_{\sqrt{2e-1}}(-u(x) - u(y))\, dx\, dy \le O(\eta) \|\Delta h\|_2^2 + O(\delta^2 \|u\|_2^2) \le O(\eta) \|\Delta h\|_2^2.
    \]
    We apply \Cref{lem:V-product} with $K = O(\eta) \|\Delta h\|_2^2$, $\Omega = C_1$, and $\mu$
    the Lebesgue measure on $\Omega$.
    (Note that the hypotheses of \Cref{lem:V-product} are satisfied when $\delta$ is small, because $|C_1| = \Theta(\delta)$,
    while $\|u\|_2^2 = O(\delta^2)$ and $\|\Delta h\|_2^2 = O(\delta^3)$).
    Since $|C_1| = \Theta(\delta)$, this gives
    \[
        \int_{C_1} V_{\sqrt{2e-1}}(-u(x))\, dx = O(\eta \|\Delta h\|_2^2 / \delta) = O(\eta \|u\|_2^2),
    \]
    as claimed.

    Now we prove the second claim.
    For $(x, y) \in C_2 \times C_1$, $\Delta h(x,y) = v_2 u(x) + w(x,y) \pm O(\delta u(y))$.
    Since $\|\delta u(y)\|_2^2 = \delta^2 \|u\|_2^2 = \Theta(\delta \|\Delta h\|_2^2)$, if we
    set $\tilde w(x,y) = w(x,y) \pm O(\delta u(y))$ then $\Delta h(x,y) = v_2 u(x) + \tilde w(x,y)$,
    and $\|\tilde w\|_2^2 \le (\eta + O(\delta)) \|\Delta h\|_2^2 = O(\eta) \|\Delta h\|_2^2$.
    Assuming that
    \[
        \int_{C_2 \times C_1} V_{2e-1}(\Delta h) \, dx\, dy \le \eta \|\Delta h\|_2^2,
    \]
    \Cref{lem:V-perturbation} implies that
    \[
        \int_{C_2 \times C_1} V_{2e-1}(v_2 u(x)) \, dx\, dy \le O(\eta) \|\Delta h\|_2^2,
    \]
    and since $|C_1| = \Theta(\delta)$, we have
    \[
        \int_{C_2} V_{(2e-1)/v_2}(u(x)) \, dx \le O(\delta \eta) \|\Delta h\|_2^2 = O(\eta \|u\|_2^2).
    \]
    The second claim follows from \Cref{lem:V-value-change} and the fact that $v_2 = \sqrt{2e-1} + O(\delta)$.
\end{proof}

\section{The improvement lemma}\label{sec:improvement}

The goal of this section is to show that if we have a bipodal graphon with approximately the expected parameter values,
then we can trade off triangle
density for entropy at a prescribed rate.
The idea behind the perturbation is simple: we
lower the triangle density by increasing the size of the small pode,
and tweaking the other parameters to keep the edge density constant.

Recall that (up to measure-preserving transformations), a bipodal graphon $h$ may be parametrized as
\[
h(x,y) = \begin{cases} a & x,y < c, \cr 
b & x,y > c, \cr 
d & x < c < y \hbox{ or } y < c < x, \end{cases}
\]
and we may assume that $c \le 1-c$.

Recall that $\calG_{e,\delta,\eta}$ is the set of bipodal graphons with edge-density $e$,
triangle density $e^3 - \delta^3$, and parameters
$a, b, c$ and $d$ satisfying $|b - e| < \eta \sqrt{\delta}$, $c < \eta$, $|d - e| < \eta$, $|a - e| > \eta$, and $|\mu| = O(\delta^{3/2})$.
The point of this definition is that our earlier estimates imply that for any $e$, some $\eta > 0$ depending on $e$,
and all sufficiently small $\delta > 0$ depending on $e$, the bipodal graphon $h$ obtained by starting with an optimal
graphon $g$ and averaging on podes belongs to $\calG_{e,\delta,\eta}$.

\begin{proposition}\label{prop:model-perturbation}
    For any graphon $h \in \calG_{e,\delta,\eta}$
    and for any $\tdens(h) > t \ge \tdens(h) - \delta^3$,
    there is a bipodal graphon $\tilde h$ with edge density $e$,
    triangle density $t$, and entropy
    \[
        S(\tilde h) \ge S(h) - C(e) \frac{2 (\tdens(h) - t)}{3\delta} (1 + O(\eta)).
    \]
\end{proposition}

Out of the four parameters $a$, $b$, $c$, and $d$, the edge-density constraint $\edens(h) = e$ can be used to eliminate one.
Defining
\[
\Delta a = a-e, \qquad \Delta d = d-e,
\qquad \mu = \frac{c \Delta a}{1-c} + \Delta d,
\]
we can change parameters to express everything in terms of $\Delta a$, $c$, and $\mu$; and
in~\eqref{delta-sum} we showed that the triangle deficit can be expressed as
\begin{equation}\label{eq:b11-reparametrization-restated}
    - \delta^3 = c^3 \alpha^3 - 6 c^3 \mu \alpha^2 + 3 c \mu^2((1 + 3 c)\alpha + e(1-c)) - (6 c^2 + 2 c^3) \mu^3,
\end{equation}
where $\alpha = {\Delta a}/({1-c})$.
To prove \Cref{prop:model-perturbation} we will simply increase $c$ while keeping $\Delta a$ and $\mu$ constant.
In terms of the original parameters, this is equivalent to setting
\begin{align*}
    c(s) &= c + s \\
    a(s) &= a \\
    d(s) &= e + \mu - \frac{c(s) \Delta a}{1 - c(s)} \\
    b(s) &= e - \frac{2 c(s) (1 - c(s)) d(s) + c^2(s) a(s)}{(1 - c(s))^2}.
\end{align*}
In particular, everything is a rational function of $s$ and is smooth for $s < 1 - c$.
Let $h_s$ be the bipodal graphon with these parameters.

If $\mu = O(\delta^{3/2})$, we see immediately from~\eqref{eq:b11-reparametrization-restated} that
\[
    \frac{d}{ds} \tau(h_s) = 3 c^2(s) \alpha^3 + O(\delta^3) = 3 \delta^2 (2e-1) + O(\delta^3 + \delta^2 s).
\]
Moreover, the $O(\delta^3)$ term is uniform over $0 \le s \le {\delta}/({2e-1})$, because $\mu$ is constant in $s$ and
$|c(s)| \le (1 + (2e-1)^{-1}) \delta$ for $s$ in this range. Since $-\delta^3$ is continuous in $s$, for any $t \in [\tdens(h) - \eta \delta^3, \tdens(h))$
there is an
\[
    s_* = \frac{\tdens(h) - t}{3 \delta^2 (2e-1)} (1 + O(\eta))
\]
such that $\tau(h_{s_*}) = t$.

Next, we consider the change in entropy as a function of $s$.
Recall that if $\int h = e$ then $\int H(h) = H(e) - \int D(h)$. 
Therefore, the change in entropy is the same as the change in $- \int D(h)$; that is,
\[
    \frac{d}{ds} S(h_s) = - \frac{d}{ds} \int D(h_s(x, y))\, dx\, dy.
\]
We can write
\[
    \int D(h_s(x, y))\, dx\, dy
    = c^2(s) D(a) + 2c(s) (1-c(s)) D(d(s)) + (1-c(s))^2 D(b(s)).
\]
Now, $D$ is differentiable with $D(e) = 0$, and $D$ is locally quadratic around $e$.
That is, $D'(e + \epsilon) = O(\epsilon)$, $D(e + \epsilon) =
O(\epsilon^2)$, and $D(1 - e + \epsilon) = D(1-e) + O(\epsilon)$.
In particular, because $|b(s) - e| = O(\eta \sqrt{\delta})$ and $|d(s) - e| = O(\eta)$,
the derivative in $s$ is bounded by the contribution of the first term:
\[
    \frac{d}{ds} \int D(h_s(x, y))\, dx\, dy = 2 c(s) c'(s) D(a) + O(\delta \eta) = 2 \frac{\delta}{2e-1} D(1-e) + O(\delta \eta),
\]
where the second equality holds for $s \le O(\eta \delta)$. Plugging in $s = s_*$ and applying the Fundamental Theorem of Calculus,
we obtain
\[
    D(h_{s_*}) = D(h) - \frac{2 \delta D(1-e)}{2e-1} (1 + O(\eta)) s_*
        = D(h) - C(e) \frac{2(\tau(h) - t)}{3\delta}(1 + O(\eta)),
\]
completing the proof of \Cref{prop:model-perturbation}.

\section{Completing the proof of bipodality}
\label{sec:completing-bipodal}

Recall that $h$ was obtained by averaging $g$ on podes. Define $\epsilon = \|g - h\|_2^2$,
and recall that $\epsilon = O(\delta^3)$. We assume for a contradiction that $g$ is optimal but not bipodal,
and hence $\epsilon > 0$.

First, note that \Cref{prop:approx-B11} and \Cref{lem:degree-concentration} imply that for any $\eta_1 > 0$,
if $\delta > 0$ is sufficiently small then $g \in \calG_{e,\delta,\eta_1}$. Indeed, part \eqref{it:approx-B11-norm}
of \Cref{prop:approx-B11} implies the required estimates on the parameters $a$, $b$, $c$, and $d$, while
\Cref{lem:degree-concentration} implies the required estimate on $\mu$.
In particular, we may apply \Cref{prop:model-perturbation} to construct a graphon $\tilde h$ with
$\tau(\tilde h) = \tau(g)$.
\Cref{lem:avg-triangle-change} implies that $\tdens(h) \le \tdens(g) + (3/2)
\epsilon \delta + o(\epsilon \delta)$, and so
\Cref{prop:model-perturbation} with $t = \tdens(g)$ gives a graphon
$\tilde h$
with $\tdens(\tilde h) = \tdens(g)$ and
\[
    S(g) \ge S(\tilde h) = S(h) - \frac 23 C(e) \frac{\tdens(h) - \tdens(g)}{\delta} (1 - O(\eta_1))
    \ge S(h) - C(e) \epsilon (1 + O(\eta_1)).
\]
In particular, for any fixed $\eta > 0$ we can find $\eta_1 > 0$ small enough so that
the conclusion of \Cref{lem:avg-entropy-change-better-2}
fails to hold for all sufficiently small $\delta$.

On the other hand, for any fixed $\eta > 0$, if $\|w\|_2^2 \ge \eta \|\Delta h\|_2^2$ then
\Cref{lem:avg-triangle-change} gives the improved bound $\tdens(h) \le \tdens(g) + (3/2) \epsilon \delta (1 - \Omega(\eta))$, meaning that (by \Cref{prop:model-perturbation}, if $\eta_1$ is small compared to $\eta$)
\[
    S(g) \ge S(h) - C(e) \epsilon (1 - \Omega(\eta)),
\]
which contradicts \Cref{lem:avg-entropy-change} (for $\delta$ sufficiently small).
We conclude that
\begin{equation}\label{eq:w-small}
    \|w\|_2^2 \le \eta \|\Delta h\|_2^2,
\end{equation}
and that the conclusion of Lemma~\ref{lem:avg-entropy-change-better-2} fails, meaning that
\begin{equation}\label{eq:V-small}
    \max\left\{\int_{C_1 \times C_1} V_{2e-1}(-\Delta h)\,dx\, dy,
    \int_{(C_1 \times C_1)^c} V_{2e-1}(\Delta h)\, dx\, dy\right\}
    \le \eta \epsilon.
\end{equation}

Finally, we will show that~\eqref{eq:w-small} and~\eqref{eq:V-small} together contradict \Cref{prop:approx-B11}.
Fix $x \in C_1$.
Recall that $v: [0, 1] \to \R$ takes two values;
let $v_1 \approx \sqrt{2e-1}$ be the value $v$ takes on $C_1$
and let $v_2 = \Theta(\delta)$ be the value $v$ takes on $C_2$.
Then $\Delta h_x = v_1 u + u(x) v + w_x$
and $g_x = h_x - \Delta h_x$, where $h_x(y) = (1 - e) + o(1)$ for $y \in
C_1$. Note that $h_x$ and $v$ are constant on $C_1$, with $h_x = 1 - e + o(1)$
and $v = \sqrt{2e-1} + o(1)$.
Recalling that $u = \|v\|_2^{-2} T_{\Delta h} v$, note that
if $u(x) \approx -\sqrt{2e-1}$ then $h_x - u(x) v \approx e$
on $C_1$ (and is constant on $C_1$).
To be precise, we consider two cases.

If $u(x) \le -(1/2)(\sqrt{2e-1} + \eta)$ then $h_x - u(x) v \ge 1/2 + \Omega(\eta)$ on $C_1$.
In particular, $h_x - u(x) v$ is closer to $e$ than it is to $1-e$, and so \cref{it:approx-B11-Omega2}
of \Cref{prop:approx-B11}
implies that $w_x$ will need to be large to compensate:
we have $g_x = h_x - u(x) v - v_1 u - w_x \ge  1/2 + \Omega(\eta) - v_1 u - w_x$
and so
by the triangle inequality,
\begin{align}
    \left(\int_{C_1} (g_x(y) - e)^2 \, dy\right)^{1/2}
    &\le |e - \frac 12 - \Omega(\eta)| + \|v_1 u + w_x\|_2 \notag \\
    &\le |e - \frac 12 - \Omega(\eta)| + \|w_x\|_2 + O(\delta).
    \label{eq:close-to-e}
\end{align}
On the other hand,
\begin{align}
    \left(\int_{C_1} (g_x(y) - (1 - e))^2 \, dy\right)^{1/2}
    &\ge |1 - e - \frac 12 - \Omega(\eta)| - \|v_1 u + w_x\|_2 \notag \\
    &\ge |1 - e - \frac 12 - \Omega(\eta)| - \|w_x\|_2 - O(\delta).
    \label{eq:close-to-1-e}
\end{align}
\Cref{it:approx-B11-Omega2} of \Cref{prop:approx-B11}
implies that~\eqref{eq:close-to-e} can be at most
$O(\delta^2)$ smaller than~\eqref{eq:close-to-1-e}.  Since $e > 1/2$, for
$\eta > 0$ sufficiently small (and $\delta > 0$ sufficiently small depending on
$\eta$),
\[
    \|w_x\|_2 \ge \Omega(\eta) - O(\delta) = \Omega(\eta) = \Omega(\eta |u(x)|).
\]

For the other case, if $u(x) \ge -(1/2) (\sqrt{2e-1} + \eta)$ then (for $\eta$ sufficiently small)
$u$ is bounded away from $-\sqrt{2e-1}$ and so
$V_{\sqrt{2e-1}}(-u(x)) \ge \Omega(u^2(x))$.

Let $A = \{x \in C_1: u(x) \le -(1/2) (\sqrt{2e-1} + \eta)\}$ (i.e. the set of $x$ for which
the first case happens), and let $B = C_1 \setminus A$. Then $\|w_x\|_2^2 \ge \Omega(\eta^2 u^2(x))$ for $x \in A$
and $u^2(x) \le O(V_{\sqrt{2e-1}}(-u(x)))$ for $x \in B$.
If $\int_A u^2(x)\, dx \ge (1/2) \int_{C_1} u^2(x)\, dx$ then
\begin{equation}\label{eq:Omega2-w-large}
    \int_{C_1} \|w_x\|_2^2 \, dx \ge \Omega(\eta^2) \int_{C_1} u^2(x)\, dx;
\end{equation}
and if $\int_B u^2(x)\, dx \ge (1/2) \int_{C_1} u^2(x)\, dx$ then
\begin{equation}\label{eq:Omega2-V-large}
    \int_{C_1} V_{\sqrt{2e-1}}(-u(x))\, dx \ge \Omega(1) \int_{C_1} u^2(x)\, dx.
\end{equation}

We will give a similar argument for $C_2$: fix $x \in C_2$ and note
that $g_x = h_x - u(x) v - v_2 u - w_x$, where $h_x \approx e$ on $C_1$
and $v \approx \sqrt{2e-1}$ on $C_1$, and $v_2 = O(\delta)$.
This time, if $u(x) \ge (1/2) (\sqrt{2e-1} + \eta)$, \cref{it:approx-B11-Omega1} of \Cref{prop:approx-B11} will imply
that $\|w_x\|_2$ is large, while in the other case we will have $V_{\sqrt{2e-1}}(u(x))$ bounded below by $u^2(x)$.

Let $A = \{x \in C_2: u(x) \ge (1/2) (\sqrt{2e-1} + \eta)\}$ and let $B = C_2 \setminus A$.
Analogously to~\eqref{eq:close-to-e} and~\eqref{eq:close-to-1-e}, if $x \in A$ then $h_x - u(x) v$ takes a constant value of at most $1/2 - \Omega(\eta)$, which
is $\Omega(\eta)$ closer to $1-e$ than it is to $e$.
Then $g_x \le 1/2 - \Omega(\eta) - v_2 u - w_x$, and so
\[
    \left(\int_{C_1} (g_x(y) - (1-e))^2\, dy\right)^{1/2} \le |1-e - \frac 12 + \Omega(\eta)| + \|w_x\|_2 + O(\delta)
\]
and
\[
    \left(\int_{C_1} (g_x(y) - e)^2\, dy\right)^{1/2} \le |e - \frac 12 + \Omega(\eta)| - \|w_x\|_2 - O(\delta),
\]
and it follows from \cref{it:approx-B11-Omega1} of \Cref{prop:approx-B11} that
\[
    \|w_x\|_2 \ge \Omega(\eta |u(x)|).
\]
On the other hand, if $x \in B$ then $V_{\sqrt{2e-1}}(u(x)) \ge \Omega(u^2(x))$ (for $\eta$ sufficiently small).
As before, depending on whether $\int_A u^2$ or $\int_B u^2$ is larger, we have either
\begin{equation}\label{eq:Omega1-w-large}
    \int_{C_2} \|w_x\|_2^2 \, dx \ge \Omega(\eta^2) \int_{C_2} u^2(x)\, dx;
\end{equation}
or
\begin{equation}\label{eq:Omega1-V-large}
    \int_{C_2} V_{\sqrt{2e-1}}(u(x))\, dx \ge \Omega(1) \int_{C_2} u^2(x)\, dx.
\end{equation}

Finally, we consider whether $\int_{C_2} u^2\, dx$ or $\int_{C_1} u^2\, dx$ is larger.
In the former case, we apply~\eqref{eq:Omega1-w-large} or~\eqref{eq:Omega1-V-large}; in the
latter case, we apply~\eqref{eq:Omega2-w-large} or~\eqref{eq:Omega2-V-large}. We conclude
that either $\|w\|_2^2 \ge \Omega(\eta^2) \|u\|_2^2$, in which case $\|\Delta h\|_2^2 = \|u\|^2_2 \|v\|^2_2 + \|w\|_2^2$ and $\|v\|_2^2 = \Theta(\delta)$ imply that
\[
    \|w\|_2^2 \ge \Omega(\epsilon),
\]
or else at least one of $\int_{C_2} V_{\sqrt{2e-1}}(u(x))\, dx$ or $\int_{C_1} V_{\sqrt{2e-1}}(u(x))\, dx$
is $\Omega(\|u\|_2^2)$, which by \Cref{lem:u-structure} implies that
\[
    \max\left\{\int_{C_1 \times C_1} V(\Delta h)\,dx\, dy, \int_{(C_1 \times C_1)^c} V(\Delta h)\, dx\, dy\right\} \ge \Omega(\epsilon).
\]
The first case contradicts~\eqref{eq:w-small}; the second contradicts~\eqref{eq:V-small}.
In either case, this contradiction implies that $\epsilon = 0$ and so $g$ is bipodal.
This concludes the proof of \Cref{thm:b11}.

\section{The region ${\calO}_2$}\label{sec:f11}

Having established the properties of the region ${\calO}_1$, where $\tdens$ is 
slightly less than $\edens^3$, 
we turn our attention to the region where the triangle density
$\tdens$ is slightly greater than $\edens^3$, the region ${\calO}_2$. In this region, the optimizing graphon was already proven to be bipodal in 
\cite{KRRS2}. To complete the proof of 
Theorem \ref{thm:f11}, we must determine the parameters of that optimal bipodal graphon. 

\begin{remark} The results of \cite{KRRS2} apply when $0 < \edens < 1/2$ as well 
as when $1/2 < \edens < 1$, as do all the computations in this section. However, they do not 
apply {\em at} $\edens=1/2$, insofar as many quantities go as negative powers of $1-2\edens$.
By contrast, our results for $\tdens < \edens^3$ only apply when $\edens > 1/2$. The situation 
when $\edens < 1/2$ and $\tdens < \edens^3$ is qualitatively different and is the subject of ongoing work. 
\end{remark}

Our analysis of ${\calO}_2$ is qualitatively similar to that of ${\calO}_1$, 
with one very important difference. 
Instead of everything being a power series in $\delta$, where $(\edens,\tdens) = (e, e^3-\delta^3)$, everything is
a power series in $\Delta \tau$, where $(\edens,\tdens) = (e,e^3+\Delta \tau)$. 
We use the constraint $\edens=e$
to express $b$ in terms of $a$, $c$, and $d$, and then use the value of 
$\tau$ to express $c$ in terms of $a$, $d$ and $\Delta \tau$. We then solve the 
equations $\partial_d S=0$ and $\partial_a S = 0$ iteratively. Knowing
$a$ and $d$ to a given order allows us to compute $c$ to 
a certain order, which then allows us to compute $a$ and $d$ to one more
power of $\Delta \tau$ than before, which allows us to compute $c$ 
to one more power of $\Delta \tau$, and the process repeats. 
We will exhibit the first few steps in this
process, enough to compute the entropy up to an $O(\Delta \tau^3)$ error.

Extended to all orders, the result would be a set of asymptotic series for 
$(a,b,c,d)$, and thus for the entropy
$S$. However, the analyticity of of parameters 
only implies that we have convergent Taylor series expansions around 
points in the interior of $\calO_2$. Our expansion in powers of $\Delta \tau$
is around a point $(\edens,\tdens)= (e,e^3)$ on the boundary 
of $\calO_2$, so convergence of our series is not guaranteed. 

We therefore need a separate
iterative method to use at fixed (non-infinitesimal) 
values of $\Delta \tau$. As before, 
we linearize the equations $\partial_a S = \partial_d S = 0$. As long as the 
Hessian of $S$ is well-approximated by a fixed matrix $M_0$ in a 
neighborhood of the approximate
values of $(a,b,c,d)$ that we have derived for $(\edens, \tdens)=(e,e^3+\Delta \tau)$, then the 
iteration 
\begin{equation}\label{notNewton}
\begin{pmatrix} a_{new} \cr d_{new} \end{pmatrix} 
= \begin{pmatrix} a_{old} \cr d_{old} \end{pmatrix}
- M_0^{-1} \begin{pmatrix} \partial_aS(a_{old}, d_{old}) \cr \partial_dS(a_{old}, d_{old})
\end{pmatrix} 
\end{equation}
is guaranteed to converge to the solution to $\partial_aS=\partial_dS=0$. 
\footnote{Using the fixed matrix $M_0$ is less efficient, but more robust, than applying Newton's method.} 

\subsection{Expressing quantities in terms of $a$ and $d$}
As before, we begin with the two identities:
\begin{equation} \label{deltab}
\Delta b = - \left ( \frac{c}{1-c} \right )^2 \Delta a - 2 \left ( \frac{c}{1-c} \right )
\Delta d
\end{equation}
and 
\begin{eqnarray}\label{eq:DeltaTau1}
\Delta \tau & = & 3ec(1-c) \left ( \frac{c}{1-c} \Delta a + \Delta d \right)^2 \cr 
&& + c^3 \Delta a^3 + 3c^2(1-c)\Delta a \Delta d^2 + 3c(1-c)^2 \Delta b \Delta d^2
+ (1-c)^3 \Delta b^3. 
\end{eqnarray} 
The difference is that now $\Delta d$ is of order 1, so the first term
in the expansion of $\Delta \tau$, 
which is $\Omega(c)$, dominates the remaining terms, which are $O(c^2)$. 
Using 
(\ref{deltab}), we have 
\begin{eqnarray*}
3c(1\!-\!c)^2\Delta b\Delta d^2 & = & -3c^3 \Delta a \Delta d^2 - 6c^2(1-c) \Delta d^3, \cr 
(1\!-\!c)^3\Delta b^3 & = & -8c^3 \Delta d^3 - \frac{12 c^4}{1\!-\!c} \Delta a \Delta d^2 
\!-\! \frac{6c^5}{(1\!-\!c)^2} \Delta a^2 \Delta d \!-\! \frac{c^6}{(1\!-\!c)^3} \Delta a^3 \cr 
3ec \left ( \frac{c}{1\!-\!c} \Delta a \!+\! \Delta d \right)^2 & = & 
3ec\left (\Delta d^2 + c \Delta d(2\Delta a - \Delta d) + \frac{c^2}{(1-c)^2} \Delta a^2 \right ), \end{eqnarray*}
which in turn gives 
\begin{eqnarray}\label{eq:DeltaTau2}
\Delta \tau & = & 3ec \Delta d^2 \cr 
&& + c^2 \left ( 6e \Delta a \Delta d - 3e \Delta d^2 + 3 \Delta a \Delta d^2- 6 \Delta d^3
\right ) \cr 
&& + c^3 \left ( \frac{3e \Delta a^2}{1-c} + \Delta a^3 - 6 \Delta a \Delta d^2 
-2 \Delta d^3 \right ) \cr 
&& - \frac{12c^4}{1-c}\Delta a \Delta d^2 - \frac{6c^5}{(1-c)^2} \Delta a^2 \Delta d
- \frac{c^6}{(1-c)^3} \Delta a^3. 
\end{eqnarray}
This expression is exact. Given estimates of $\Delta d$ to order $\Delta \tau^k$ and 
$\Delta a$ to order $\Delta \tau^\ell$, it determines $c$ to order $\Delta \tau^{k+1}$ or 
$\Delta \tau^{\ell+2}$, whichever is larger. As long as the estimates of $\Delta a$ and 
$\Delta d$ only involve integer powers of $\Delta \tau$, so will the estimates of 
$c$. 

In particular, given values of $\Delta a$ and $\Delta d$, we have 
\[
c  =  \frac{\Delta \tau}{3d \Delta d^2} - c^2 \left ( \frac{2 \Delta a}{\Delta d} - 1
+ \frac{\Delta a - 2 \Delta d}{e} \right ) + O(\Delta \tau^3)
\]
The fact that $c$ can be expanded as a polynomial in $c$ whose coefficients are rational
functions of $a$ and $d$ means that we can compute $\partial_a c$ and $\partial_d c$ to
within $O(\Delta \tau^3)$ by taking the derivative of this expression. Differentiating implicitly,
we have 
\begin{eqnarray*} 0 & = & \partial_a c \left ( 
3e \Delta d^2 + 6c(2e\Delta a \Delta d - e\Delta d^2 + \Delta a \Delta d^2 - 2 \Delta^2)
+ O(c^2) \right ) \cr 
&& + c^2 (6e \Delta d + 3 \Delta d^2) + O(\Delta \tau^3),
\end{eqnarray*}
which gives 
\[
\partial_a c = -c^2 \left ( \frac{2}{\Delta d} + \frac{1}{e} \right ) + O(\Delta \tau^3).
\]
Similarly, 
\begin{eqnarray*} 0 & = & \partial_d c \left ( 
3e \Delta d^2 + 6c(2e\Delta a \Delta d - e\Delta d^2 + \Delta a \Delta d^2 - 2 \Delta \tau^2)
+ O(c^2) \right ) \cr 
&& + 6ec \Delta d + 6c^2(e \Delta a -e\Delta d + \Delta a \Delta d - 3 \Delta d^3 ) 
+ O(\Delta \tau^3).
\end{eqnarray*}
After a little algebra, this yields 
\[ \partial_d c = \frac{-2c}{\Delta d} \left ( 1 + c \left ( 
1 - 3 \frac{\Delta a}{\Delta d} - \frac{\Delta a}{e} + \frac{\Delta d}{e} \right ) 
\right ) + O(\Delta \tau^3).
\] 
Taking the derivative of (\ref{deltab}) with respect to $a$ and $d$ then gives 
\begin{eqnarray*}
\partial_a b & = & c^2 \left ( 3 + 2 \frac{\Delta d}{e} \right ) + O(\Delta \tau^3), \cr 
\partial_d b & = & 2c\left ( 1 + c \left ( 5 - 4 \frac{\Delta a}{\Delta d} - 2 
\frac{\Delta a}{e} + 2 \frac{\Delta d}{e} \right ) \right ) + O(\Delta \tau^3). 
\end{eqnarray*}

\subsection{Solving $\partial_a S = \partial_d S = 0$}
We now compute the partial derivatives of $S$, beginning with 
$\partial_d S$ to order $\Delta \tau$, obtaining: 
\[
\partial_d S = 2c H'(d) + 2c H'(b) - \frac{4c}{\Delta d} (H(d)-H(b)) + O(\Delta \tau^2),
\]
Setting this equal to zero and dividing by $2c$ gives 
\[
H'(d) + H'(b) - \frac{2}{\Delta d} (H(d)-H(b)) = O(\Delta \tau), 
\]
implying that 
\[
d = 1-b + O(\Delta \tau) = 1-e + O(\Delta \tau). 
\]
We therefore write 
\[
d = 1-e + d_1, \qquad \Delta d = 1-2e + d_1, 
\]
where $d_1$ is a quantity of order $\Delta \tau$. This allows us to express quantities like
$H(d)$, $H'(d)$, $H''(d)$, $\Delta b$, $H(b)$, $H'(b)$ and $H''(b)$ in terms of $d_1$. 

We next compute $\partial_a S$ to order $\Delta \tau^2$.
\begin{eqnarray*} 
\partial_a S & = & c^2 H'(a) + (1-c)^2 \partial_a b H'(b) \cr 
&& + 2 \partial_a c (H(d)-H(b) + c (H(a)+H(b)-2H(d))) \cr 
& = & c^2 H'(a) + c^2 \left ( 3 + \frac{2 \Delta d}{e} \right ) H'(b) + O(\Delta \tau^3).
\end{eqnarray*}
Setting $\partial_a S$ equal to zero gives 
\begin{eqnarray*}
H'(a) & = & - \left ( 3 +\frac{2 \Delta d}{e} \right ) H'(b) + O(\Delta \tau)\cr
& = & - \left ( 3 + \frac{2 (1-2e)}{e} \right ) H'(e) + O(\Delta \tau) \cr
& = & \left ( 1 - \frac{2}{e} \right ) H'(e) + O(\Delta \tau). 
\end{eqnarray*}
Let $a_0$ be the solution to 
\[
H'(a_0) = \left (1-\frac{2}{e} \right ) H'(e), \]
which happens to be 
\[
a_0 = \left (1 + \left ( \frac{e}{1-e} \right )^{\frac{2}{e}-1} \right)^{-1}.
\]
We then have $a = a_0 + O(\Delta \tau)$.

To complete the proof of Theorem \ref{thm:f11}, we must show that there is an 
iterative scheme for approximating $(a,b,c,d)$, and therefore $S$, for fixed values of 
$(e,\Delta \tau)$ with $\Delta \tau$ sufficiently small. As with our analysis of
${\calO}_1$, 
we do this by computing the Hessian of $S$. 

Our previously computed first derivatives of $b$ and $c$ are, to leading 
order:
\begin{eqnarray*}
\partial_a b & = & c^2 \left (3 + \frac{2\Delta d}{e} \right ) + O(\Delta \tau^3), \cr 
\partial_ac & = & -c^2 \left (\frac{2}{\Delta d} + \frac{1}{e} \right ) + O(\Delta \tau^3), \cr 
\partial_d b & = & 2c + O(\Delta \tau^2), \cr 
\partial_d c & = & - \frac{2c}{\Delta d} + O(\Delta \tau^2). 
\end{eqnarray*}

Our second partials are:
\begin{eqnarray*}
\partial^2_{aa} b & = &  O(\Delta \tau^3), \cr 
\partial^2_{ad} b & = & - 6 \frac{c^2}{e} - 12 \frac{c^2}{\Delta d} + O(\Delta \tau^3), \cr 
\partial^2_{dd} b & = & -4 \frac{c}{\Delta d} + O(\Delta \tau^2), \cr 
\partial^2_{aa} c & = & O(\Delta \tau^3), \cr 
\partial^2_{ad} c & = & \frac{c^2}{\Delta d^2} \left (10 + \frac{4\Delta d}{e} \right ) + O(\Delta \tau^3), \cr 
\partial^2_{dd} c & = & 6 \frac{c}{\Delta d^2} + O(\Delta \tau^2).
\end{eqnarray*}

Next we compute the partial derivatives of $S = c^2H(a)+2c(1-c)H(d)+(1-c)^2H(b)$. 
Since $\partial_a b$ and $\partial_a c$ are $O(\Delta \tau^2)$, we have 
\[
\partial_a S = \partial_a b H'(b) + 2 \partial_a c (H(d)-H(b)) + c^2 H'(a) + O(\Delta \tau^3).
\]
Taking another derivative with respect to $a$, 
only one term contributes at order $\Delta \tau^2$:
\[
\partial^2_{aa} S = c^2 H''(a_0) + O(\Delta \tau^3). 
\]
If instead we take a derivative of $\partial_aS$ with respect to $d$, three terms contribute at order $\Delta \tau^2$:
\begin{eqnarray*}
\partial^2_{ad} S & = & 2c \partial_d c H'(a) + \partial^2_{ad}b H'(b) + 2 \partial_d c H'(d) + O(\Delta \tau^3) \cr 
& = & - \frac{4c^2}{\Delta d} H'(a) - 6c^2 \left ( \frac{1}{e} + \frac{2}{\Delta d}\right  )
H'(b) - 2c^2 \left ( \frac{1}{e} + \frac{2}{\Delta d}\right  )H'(d) + O(\Delta \tau^3). 
\end{eqnarray*} 
Using the fact that $H'(b) \approx H'(e)$, $H'(d) \approx -H'(e)$ and 
$H'(a) \approx \left ( 1 - {2}/{e} \right )H'(e)$, this simplifies to 
\[ 
\partial^2_{ad} S = \frac{4c^2(1-e)}{e(1-2e)} H'(e) + O(\Delta \tau^3).
\]
Finally, we compute $\partial^2_{dd} S$. We have 
\begin{eqnarray*} 
\partial_d S &=& \partial_d b H'(b) + 2c H'(d) + 2 \partial_d c(H(d)-H(b)) + O(\Delta \tau^2).\cr
\partial^2_{dd} S & = & \partial^2_{bb} b H'(b) + 4 \partial_d c H'(d) + 2c H''(d) 
+ 2 \partial^2_{dd} c (H(d) - H(b)) + O(\Delta \tau^2) \cr 
& = & -\frac{4c}{\Delta d} H'(b) - \frac{8c}{\Delta d} H'(d) + 2c H''(d) + O(\Delta \tau^2) \cr
& = & 2c \left ( H''(e) + \frac{2}{1-2e} H'(e) \right ) + O(\Delta \tau^2). 
\end{eqnarray*} 
The upshot is that $\partial^2_{aa} S$ and $\partial^2_{ad}S $ are nonzero constants times
$c^2$ plus $O(\Delta \tau^3)$, while $\partial^2_{dd} S$ is a nonzero constant times 
$c$ plus $O(\Delta \tau^2)$. Replacing $c$ with $\Delta \tau/ (3e (1-2e)^2)$ and 
throwing away the error terms, we obtain a non-singular 
matrix $M_0$, with no explicit dependence on $a$ or $d$, that comes
within $$\begin{pmatrix} O(\Delta \tau^3) & O(\Delta \tau^3) \cr
  O(\Delta \tau^3) & O(\Delta \tau^2) \end{pmatrix}$$ of the actual
Hessian for all $(a,d)$ within $O(\Delta \tau)$ of  
$(a_0, 1-e)$. In particular the iteration
(\ref{notNewton}), beginning at $(a,d) = (a_0, 1-e)$, converges to the 
unique solution to $\partial_a S=\partial_d S = 0$ 
for all sufficiently small
values of $\Delta \tau$. This completes the proof of Theorem \ref{thm:f11}.

\section{Entropy estimates}\label{sec:entropy}

Having established the properties of the optimizing graphons in
${\calO}_1$ and ${\calO}_2$
we now consider the asymptotic values of the the entropy function as $\tdens$ approaches $\edens^3$, both from above and from below. 

\subsection{The ${\calO}_2$ region with $\edens=e$, $\tdens>e^3$, $e \ne 1/2$.}

Let \begin{eqnarray} \label{c12}
c_1 & = & \frac{\Delta \tau}{3e(1-2e)^2} \cr 
c_2 & = & -2 \frac{c_1 d_1}{1-2e} -c_1^2 \left ( \frac{2(a_0-e)-(1-2e)}{1-2e} + \frac{a_0-e -2(1-2e)}{e}
\right ). 
\end{eqnarray}
By our previous estimates, we have 
\[ c = c_1+ c_2 + O(\Delta \tau^3). \]
We estimate the terms in (\ref{WhatsS}) using 
\begin{eqnarray*}
H(b) & = & H(e) + H'(e) \Delta b + \frac{H''(e)}{2} \Delta b^2 + O(\Delta \tau^3) \cr 
& = & H(e) - H'(e) \left ( \frac{2c}{1-c}(1-2e+d1) + c^2 \Delta a
\right ) + 2 c^2 (1-2e)^2 H''(e)  \\
& = & H(e) -2c(1-2e)H'(e) -2cd_1 H'(e) + c^2(2 H''(e)(1-2e)^2 \cr && -H'(e)(a_0-e + 2(1-2e)))
+ O(\Delta \tau^3) \cr 
H(d) - H(b) & = & H'(1-e) d_1 - H'(e) \Delta b + O(\Delta \tau^2) \cr 
& = & H'(e) (2c(1-2e)-d_1) + O(\Delta \tau^2) \cr 
H(a)+H(b)-2H(d) & = & H(a) - H(e) + O(\Delta \tau).\cr
\end{eqnarray*} 
Plugging these terms into (\ref{WhatsS}) yields 
\begin{eqnarray*}
S & = & H(e) - 2(c_1+c_2)(1-2e+d_1)H'(e) - 2c_1d_1H'(e) \cr 
&& + c_1^2 \left ( H(a_0)-H(e) + H'(e) (2(1-2e) -(a_0-e)) + 2H''(e)(1-2e)^2 \right )
+ O(\Delta \tau^3)
\end{eqnarray*}
Substituting for $c_2$ from (\ref{c12}), we find that the $c_1d_1$ terms cancel,
leaving us with 
\begin{eqnarray*}
S & = & H(e) \!-\!2c_1(1\!-\!2e)H'(e) + 2c_1^2H'(e) \left ( 2(a_0\!-\!e) - (1\!-\!2e) + \frac{(1\!-\!2e)(a_0\!-\!e \!-\! 2(1\!-\!2e))}{e} \right ) \cr 
&& + c_1^2 \left ( H(a_0)-H(e) + H'(e) (2(1-2e) -(a_0-e))  + 2H''(e)(1-2e)^2 \right )
+ O(\Delta \tau^3) \cr 
& = & H(e) - 2c_1(1-2e) H'(e) \cr 
&& + c_1^2 \left ( H(a_0) - H(e) + H'(e) \left ( 3(a_0-e) +\frac{2(1-2e)}{e}(a_0+3e-2) + 2 H''(e) (1-2e)^2
\right ) \right ) \cr && + O(\Delta \tau^3) \cr 
&= & H(e) - \frac{2\Delta \tau}{3e(1-2e)} H'(e) \cr 
&& + \frac{\Delta \tau^2}{9e^2(1\!-\!2e)^4} \left ( H(a_0) \!-\! H(e) \!+\! H'(e) \left ( 3(a_0\!-\!e) \!+\!\frac{2(1\!-\!2e)}{e}(a_0\!+\!3e-2)
\right )\right ) \cr && + \frac{2 \Delta \tau^2 H''(e)}{9e^2(1-2e)^2} + O(\Delta \tau^3). \cr &&
\end{eqnarray*}

The linear coefficient is of course negative, since $H'(e)$ is positive for $e<
1/2$ and is negative for $e > 1/2$. The quadratic coefficient is more 
complicated but is also negative, diverging as 
$e^{-3} \ln(e)$ as $e \to 0$ and as $(e-1)^{-1}$ as $e \to 1$. See
Figure \ref{FIG:S2}.

\begin{figure}[ht]
\centering
\includegraphics[width=4in]{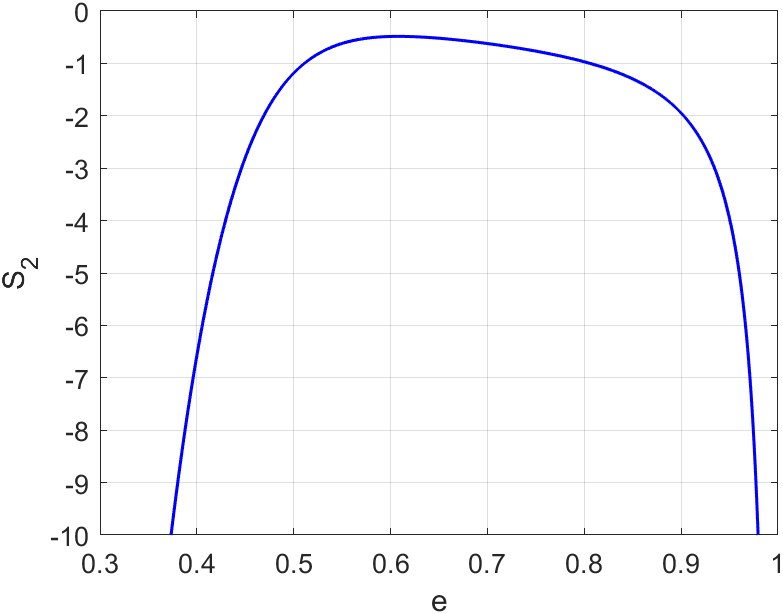}
\caption{The coefficient of $\Delta \tau^2$ as a function of $e$.}
\label{FIG:S2}
\end{figure}



\subsection{The region ${\calO}_1$ with $\edens=e$, $\tdens < e^3$, $e > 1/2$}

As before, we express the parameters $(b,c,d)$ using the coordinates $(a,\mu)$. 
For convenience, we repeat the key identities:
\begin{eqnarray}\label{cheat-sheet}
\Delta b & = & \frac{c}{1-c}\left ( \frac{c \Delta a}{1-c} - 2\mu \right ), \cr 
\Delta d & = & \mu - \frac{c \Delta a}{1-c}, \cr 
\frac{c \Delta a}{1-c} & = & - \delta + \frac{2 \mu \delta^2 - \mu^2 e}{(2e-1)\delta}
+ O(\delta^4), \cr 
c & = & \frac{\delta + \frac{\mu^2e - 2 \mu \delta^2}{(2e-1)\delta}}{2e-1+\delta-a_1}
+ O(\delta^4),
\end{eqnarray}
where $a=1-e+ a_1$. 

Since $S=H(b) + 2c(H(d)-H(b)) + c^2(H(a)+H(b)-2H(d))$, we need to evaluate $H(b)$ 
through $O(\delta^4)$, $H(d)$ through $O(\delta^3)$ and $H(a)$ through $O(\delta^2)$:
\begin{eqnarray*}
H(b) & = & H(e) + H'(e) \Delta b + \frac12 H''(e) \Delta b^2 + O(\delta^6) \cr 
H(d) & = & H(e) + H'(e) \Delta d + \frac12 H''(e) \Delta d^2 + \frac16 H'''(e)
\Delta d^3 + O(\delta^4) \cr 
H(a) & = & H(1-e) + H'(1-e)a_1 + \frac12 H''(1-e)a_1^2 + O(\delta^3) \cr 
& = & H(e) - H'(e) a_1 + \frac12 H'(e) a_1^2 + O(\delta^3).
\end{eqnarray*}
This makes 
\begin{eqnarray} \label{eq:Ssubtotal}
S & = & H(e) + H'(e) \left ( (1-c)^2 \Delta b + 2c(1-c) \Delta d - c^2 a_1\right ) \cr 
&&+ \frac{H''(e)}{2} \left ( (1-c)^2 \Delta b^2 + 2c(1-c)\Delta d^2 + c^2 a_1^2 \right ) \cr 
&& + \frac{H'''(e)}{3} c(1-c) \Delta d^3 + O(\delta^5). 
\end{eqnarray}

The coefficient of $H'(e)$ is 
$$
c(1-c)\left (\frac{c\Delta a}{1-c} - 2\mu \right ) + 2c(1-c) \left (\mu- \frac{c\Delta a}{1-c} \right ) - c^2 a_1 = c^2(2e-1 -2 a_1). 
$$ 
Squaring and expanding the formula for $c$ in (\ref{cheat-sheet}) gives 
\begin{eqnarray*} 
c^2 &=& \frac{\delta^2 + 2 \frac{\mu^2 e - 2\mu \delta^2}{2e-1}}{(2e-1+\delta - a_1)^2} + 
O(\delta^5) \cr
& = & \frac{\delta^2}{(2e-1)^2} - \frac{2 \delta^2(\delta-a_1)}{(2e-1)^3} + 
\frac{3 \delta^2(\delta-a_1)^2}{(2e-1)^4} + \frac{2(\mu^2e-2\mu\delta^2)}{(2e-1)^3}
+ O(\delta^5). 
\end{eqnarray*}
Multiplying by $(2e-1-2a_1)$, and using the fact that $a_1=-\delta + O(\delta^2)$, the coefficient of $H'(e)$ becomes 
$$
\frac{\delta^2}{(2e-1)} - \frac{2 \delta^3}{(2e-1)^2} + 
\frac{4 \delta^4}{(2e-1)^3} + \frac{2(\mu^2e-2\mu\delta^2)}{(2e-1)^2}
+ O(\delta^5). 
$$

Next we consider the coefficient of $(1/2) H''(e)$, namely 
\begin{eqnarray*} 
(1-c)\Delta b^2 + 2c(1-c) \Delta d^2 + c^2 a_1^2 & = & 
c^2 \left ( \frac{c \Delta a}{1-c} - 2\mu \right )^2 + 2c(1-c) \left ( 
\mu - \frac{c\Delta a}{1-c} \right )^2 + c^2 a_1^2 \cr 
& = & (c^2+2c(1-c)) \left ( \frac{c\Delta a}{1-c} \right )^2 - 4\mu c^2 \Delta a + c^2 a_1^2.
\end{eqnarray*}
Using the third identity of (\ref{cheat-sheet}), this reduces to 
$$
(2c-c^2) \delta^2 + 4\mu c^2(2e-1) + c^2 a_1^2 + O(\delta^5).
$$
But $a_1 \approx - \delta$ and 
$$ 
2c-c^2 = \frac{2 \delta}{2e-1} -5 \frac{\delta^2}{(2e-1)^2} + O(\delta^3), 
$$
so this becomes 
$$ 
\frac{2 \delta^3}{2e-1} - \frac{4 \delta^4}{(2e-1)^2} + \frac{4 \mu \delta^2}{2e-1}
+ O(\delta^5).
$$

Combining everything gives 
\begin{eqnarray}\label{eq:Stotal}
S & = & H(e) + \frac{\delta^2}{2e-1} H'(e) + \delta^3 \left ( 
\frac{H''(e)}{2e-1} - \frac{2 H'(e)}{(2e-1)^2} \right ) \cr 
&& + \delta^4 \left ( \frac{4 H'(e)}{(2e-1)^3} - \frac{2H''(e)}{(2e-1)^2} + \frac{H'''(e)}{3(2e-1)} \right ) \cr 
&& + \mu \delta^2 \left ( \frac{2H''(e)}{2e-1} - \frac{4H'(e)}{(2e-1)^2} \right ) 
+ \frac{2 \mu^2 e H'(e)}{(2e-1)^2} +O(\delta^5).
\end{eqnarray}
This is a quadratic function of $\mu$ that is maximized at
$$ \mu = \frac{\delta^2}{e H'(e)} \nu + O(\delta^3), $$
where 
$$
\nu = H'(e) - \left ( e - \frac12 \right ) H''(e), 
$$
in agreement with (\ref{WhatsMu}). 
Our final total, involving only the parameters $e$ and $\delta$, is then 
\begin{eqnarray*}
S &=&  H(e) + \frac{\delta^2}{2e-1}H'(e) - \frac{2\delta^3 \nu }{(2e-1)^2} \cr 
&& + \delta^4 \left ( \frac{H'''(e)}{3(2e-1)} + \frac{4 \nu}{(2e-1)^3} - \frac{2\nu^2}{e H'(e) (2e-1)^2} \right ) + O(\delta^5).
\end{eqnarray*}

Although this formula for the entropy describes a smooth function of 
$\delta$ at $\delta=0$, it is not a smooth function of $\tau = e^3-\delta^3$. The first and second derivatives of 
$s(\edens,\tdens)$ with respect to $\tdens$ are positive and diverge as 
$\delta^{-1}$ and $\delta^{-4}$, respectively, as $\delta \to 0$. 

Comparing the entropy functions in the regions ${\calO}_1$ and ${\calO}_2$, we see that 
the entropy is continuous in $\tdens$ at $\tdens=\edens^3$, but the first
derivative is not. The first derivative
approaches a finite value as $\tdens$ approaches $\edens^3$ from above, as does the second derivative, but these quantities diverge as
$\tdens$ approaches $\edens^3$ from below. This 
completes the proof of Theorem \ref{thm:entropy}.

\section{Longer cycles}\label{sec:kgon}

Finally, we turn to the edge-$k$-cycle problem for odd $k>3$
and prove Theorem \ref{thm:kgon}. 
The strategy is essentially the same as in the proof of 
Theorems \ref{thm:b11}--\ref{thm:entropy}. The part of 
the proof concerning $\tdens_k<e^k$ follows 
these steps:
\begin{enumerate}
\item We show that, for any graphon, $\tdens_k$ is bounded below
by $e^k - \|\delta g\|_2^k$, with equality only when 
$\delta g$ has rank 1 with eigenvalue $-\delta$ 
and the degree function is constant. Combined with our existing
estimates of a graphon with $\|\delta g\|_2$ of fixed size, 
this gives upper bounds on the entropy of any graphon with 
the given values of $e$ and $\delta$. These bounds, as 
a function of $e$ and $\delta$, are identical
to those obtained for triangles. 

\item A bipodal graphon with $a=1-e$ and $d=1+\delta$ comes 
within $O(\delta^3)$ of achieving that upper bound. This 
bounds the extent to which an optimal graphon can differ from 
our model graphon. These bounds are sharper than in the edge-triangle model. In the edge-triangle model, we had 
the a priori estimate 
$\int_0^1 \delta d(x)^2 \, dx = O(\delta^4)$, but with
$k$-cycles we have $\int_0^1 \delta d(x)^2 \, dx = O(\delta^{k+1})$.
When dealing with bipodal graphons, this means that the parameter
$\mu$ is a priori $O(\delta^{k/2})$ instead of $O(\delta^{3/2})$.
(The $\mu$ parameter will eventually turn out to be $O(\delta^{k-1})$.)  We also obtain
bounds on how far $\delta g$ is (in an $L^2$ sense) from having rank one. These bounds are qualitatively similar to those 
derived for the edge-triangle model. 

\item The argument that $g$ is approximately bipodal, taking on
values that are close to $e$ or $1-e$, proceeds exactly as before. 

\item The averaging argument is similar to before. If the
optimal graphon weren't already bipodal, we could average its 
value over each quadrant to get a bipodal graphon with the same value
of $\edens$, and a quantifiable trade-off between $\tdens_k$
and entropy. Moreover, a sharp trade-off would imply the same
structural properties of $g$ that we had in \Cref{sec:averaging}.

\item If we knew that the optimal graphon were bipodal, we could
maximize the entropy as in Section \ref{sec:entropy}, expressing
all parameters in terms of $\Delta a$ and $\mu$. There is 
a positive term proportional to $\delta^2 \mu$, as in the 
edge-triangle model, and a negative term proportional to 
$\mu^2 \delta^{3-k}$. The optimal value of $\mu$ is then 
$O(\delta^{k-1})$, implying that $d = e + \delta +O(\delta^{k-1})$ 
and the contribution of $\mu$ to the entropy is $O(\delta^{k+1})$.
This shows that there is at most one possiblity for a bipodal optimizer,
and gives good estimates for its parameters.

\item Combining the last two steps, if the original graphon were not
optimal, we could replace it with the averaged graphon of Step 3,
then perturb the bipodal parameters using the estimates from Step 4
to construct a graphon with the same $\edens$ and $\tdens_k$, but larger
entropy.

\end{enumerate}

In the limit as $k \to \infty$, our problem reduces to finding 
a graphon of the form $g(x,y) = e -\delta v(x)
v(y)$, where $\int_0^1 v(x)\, dx = 0$ and 
$\int_0^1 v(x)^2\, dx=1$, that maximizes the entropy for fixed
$e$ and $\delta$. This graphon is bipodal, with $d$ and 
${-c\Delta a}/({1-c})$ exactly
equal to $\delta$. 
Theorem \ref{thm:kgon} only gives $a = 1-e-\delta + O(\delta^2)$
to first order in $\delta$, but there is no obstacle to computing
higher coefficients. The parameters for the 
optimal graphon for any fixed $k$ are all within $O(\delta^{k-1})$
of the parameters for this limiting graphon. 

We now begin the proof. 

\medskip

\noindent{\bf Step 1:} Writing $g(x,y)=e+\delta g(x,y)$, we have 
\[
\tdens_k = \Tr(g^k) = \Tr((e+\delta g)^k).
\]
Expanding this out gives $e^k + \Tr(\delta g^k)$ plus a number 
of cross terms. Some of the cross terms are of the form 
$e^{k-m} \langle 1, T_{\delta g^m} 1 \rangle$, where 
1 denotes the constant function in $L^2([0,1])$. Others are 
powers of $e$ times the product of several $\langle 1, T_{\delta g^m}1
\rangle$ expressions. 

The expression $\langle 1,T_{\delta g^m} 1 \rangle$  is zero if $m=1$, is
$\int_0^1 \delta d(x)^2 \, dx$ if $m=2$, and is bounded
by $\|\delta g\|_{op}^{m-2} \int_0^1 \delta d(x)^2 \, dx$ if
$m>2$. 
Since $\delta g$ (viewed as an integral operator) 
is self-adjoint,
its operator norm is the size of its largest eigenvalue, which 
is bounded by its $L^2$ norm, which is small compared to 1. 
Of course, the product of two or more expressions of the form 
$\langle 1, T_{\delta g^m} 1\rangle $ is smaller still. 

The upshot is that 
\begin{eqnarray*}
\tdens_k & = &  e^k + \Tr(\delta g^k) + k e^{k-2}
\int_0^1 \delta d(x)^2 \, dx + \cr 
&  & + \hbox{terms small compared to} \int_0^1 \delta d(x)^2 \, dx \cr 
& \ge & e^k + \Tr(\delta g^k) + \frac{ke^{k-2}}{2} 
\int_0^1 \delta d(x)^2 \, dx.
\end{eqnarray*}
The last term is of course positive-definite. This implies that
$\| \delta g \|_2 \ge \delta$, with equality (if and) only if 
$\delta g$ is rank-1 with eigenvalue $-\delta$ and 
$\int_0^1 \delta d(x)^2 \, dx =0$. 

We have previously shown that any graphon with 
$\|\delta g\| \ge \delta$ has entropy at most 
\[ H(e) + \frac{\delta^2 H'(e)}{2e-1} + O(\delta^3),
\]
so this is also an upper bound on the entropy of any graphon
with $\tdens_k = e^k-\delta^k$. 

\smallskip

\noindent{\bf Step 2:} Consider the bipodal graphon with 
$a=1-e$, $b=e-{\delta^2}/{(2e-1})$, 
$c = {\delta}/({2e-1 + \delta})$ and $d=e+\delta$. This graphon
has edge density $e$, $k$-cycle density $e^k-\delta^k$, and 
entropy $H(e) + {\delta^2 H'(e)}/({2e-1}) + O(\delta^3)$. 
That is, it comes within $O(\delta^3)$ of saturating
the upper bound. 

This means that the errors in step 1, both those involving 
$\int_0^1 \delta d(x)^2 \, dx$ being nonzero and those from comparing 
$\Tr(\delta g^k)$ to $\|\delta g\|_2^k$, cannot cost us more 
than $O(\delta^3)$ in entropy, or equivalently cannot cause 
$\|\delta g\|_2^2$ to be greater than $\delta^2 + O(\delta^3)$. 

In particular, since an $O(\delta^3)$ change in $\| \delta g\|_2^2$
corresponds to an $O(\delta^{k+1})$ change in $\|\delta g\|_2^k$,
we must have 
$$ \int_0^1 \delta d(x)^2 \, dx = O(\delta^{k+1}).$$
Furthermore, the contribution of the small eigenvalues of 
$\delta g$ to $\| \delta g\|_2^2$ is at most $O(\delta^3)$. 
That is, $\delta g$ is within $O(\delta^{3/2})$, in an $L^2$-sense,
of being rank-1. 

\smallskip

\noindent{\bf Step 3:}
Approximate bipodality (i.e., \Cref{prop:approx-B11}) proceeds exactly as in
the $k=3$ case, because the only ingredients in the proof of \Cref{prop:approx-B11}
were the estimates from Step 1. The error terms in \Cref{prop:approx-B11}
are the same order in $\delta$ for every $k$.

\smallskip

\noindent{\bf Step 4:}
Let $h$ be the graphon obtained by averaging $g$ on podes, and write $g = h + \Delta h$.
We need to estimate the changes to entropy and $\tau_k$ in
going from $g$ to $h$, as was done in \Cref{sec:averaging} for $k=3$.
The change in entropy follows exactly as in \Cref{lem:avg-entropy-change-better-2}:
\[
    S(g) \le S(h) - C(e) \|\Delta h\|_2^2,
\]
and the inequality is sharp if and only if the non-zero values of $\Delta h$ are close to $1-2e$ on $C_1 \times C_1$
and $2e-1$ elsewhere.

To bound the change in $\tau_k$, note that $h$ (being bipodal) has rank two, and that its eigenvalues
are approximately $\lambda_1(h) \approx e$ and $\lambda_2(h) \approx -\delta$. Since $\Delta h$ integrates to zero on each pode, $\inr{w}{T_{\Delta h} w} = 0$
for every non-trivial eigenfunction $w$ of $h$.
It follows that the eigenvalues of $g = h + \Delta h$ are
\begin{align*}
    \lambda_1(g) &= \lambda_1(h) + \frac{|\inr{u_1}{T_{\Delta h} u_2}|^2}{e}(1 + O(\delta)) \ge \lambda_1(h) \\
    \lambda_2(g) &= \lambda_2(h) - \frac{|\inr{u_1}{T_{\Delta h} u_2}|^2}{\delta}(1 + O(\delta)) \ge \lambda_2(h) - \frac{\|\Delta h\|_2^2}{2\delta}(1 + O(\delta)),
\end{align*}
where $u_1$ and $u_2$ are the non-trivial (normalized) eigenfunctions of $h$.
Moreover, the second inequality is sharp if and only if $\Delta h(x,y) = u_2(x) T_{\Delta h} u_2(y) + u_2(y) T_{\Delta h} u_2(x) + w(x,y)$ for some remainder $w$ with
$\|w\|_2 = o(\|\Delta h\|_2)$.
All other eigenvalues of $h$ are zero, and so all other eigenvalues of $g$ are bounded in absolute value by $\|\Delta h\|_2 = O(\delta^{3/2})$, and also the sum of their squares is at most $\|\Delta h\|_2^2$.
In particular, $\sum_{i \ge 3} |\lambda_i(g)|^k \le \|\Delta h\|_2^{k-2} \sum_i |\lambda_i(g)|^2 \le \|\Delta h\|_2^k$.
Since $\tau_k(h) = \lambda_1^k(h) + \lambda_2^k(h)$ and $\tau_k(h) = \sum_i \lambda_i^k(g)$, we have
\begin{align*}
    \tau_k(h)
    &\le \lambda_1^k(g) + \left(\lambda_2(g) + \frac{\|\Delta h\|_2^2}{2\delta}(1 + O(\delta))\right)^k + \sum_{i \ge 3} |\lambda_i(g)|^k \\
    & \le \lambda_1^k(g) + \lambda_2^k(g) - k \frac{\delta^{k-2} \|\Delta h\|_2^2}{2}(1 + O(\delta)) + \sum_{i \ge 3} |\lambda_i(g)|^k \\
    & \le \tau_k(g) - k \frac{\delta^{k-2} \|\Delta h\|_2^2}{2}(1 + O(\delta)),
\end{align*}
and the bound is sharp if and only if $\Delta h \approx u_2(x) T_{\Delta h} u_2(y) + u_2(y) T_{\Delta h} u_2(x)$.
This is the analogue of \Cref{lem:avg-triangle-change} for general $k$.

\smallskip

\noindent{\bf Step 5:} When considering bipodal graphons,
we can compute the two eigenvalues of $g$ from the 
trace and the Hilbert-Schmidt norm: 
\begin{eqnarray*}
\lambda_1 + \lambda_2 & = & e + c \Delta a + (1-c)\Delta b \cr 
& = & e + \frac{d \Delta a}{1-c} - 2 \mu c \cr 
\lambda_1^2 + \lambda_2^2 & = & c^2 a^2 + 2c(1-c)d^2 + (1-c)^2b^2 \cr 
& = & e^2 + c^2 \Delta a^2 + 2c(1-c)\Delta d^2 + (1-c)^2 \Delta b^2\cr
& = & e^2 + \left ( \frac{c\Delta a}{1-c}\right)^2 -4c\mu
\left ( \frac{c\Delta a}{1-c}\right ) + 2\mu^2(c+c^2).
\end{eqnarray*}
From this we compute the two eigenvalues
\[
\lambda_{1,2} = \frac12 \left ( 
e + \frac{c\Delta a}{1-c} - 2\mu c \pm 
\left ( e - \frac{c\Delta a}{1-c} + 2\mu c \right ) 
\sqrt{1 + \frac{4\mu^2c(1-c)}{\left ( e - \frac{c\Delta a}{1-c} + 2\mu c \right ) ^2}}\right ).
\]
These work out to 
\begin{eqnarray}
\lambda_1 & = & e + \frac{\mu^2c(1-c)}{e+2\mu c - \frac{c\Delta a}{1-c}} + O(\mu^4\delta^2) \label{eq:k-cycle-lambda1} \\
\lambda_2 & = & \frac{c\Delta a}{1-c} - 2\mu c - 
\frac{\mu^2c(1-c)}{e+2\mu c - \frac{c\Delta a}{1-c}} + O(\mu^4\delta^2) \label{eq:k-cycle-lambda2}
\end{eqnarray}
Note that if $\mu=0$, then our eigenvalues are exactly $e$ and 
$\frac{c \Delta a}{1-c}$, so we have 
$e^k-\delta^k = \lambda_1^k+\lambda_2^k = e^k + \left ( 
\frac{c\Delta a}{1-c}\right )^k$, so $\frac{c\Delta a}{1-c}=-\delta$. We are quantifying how much they differ from
these values when $\mu$ is nonzero. However, we know that 
$\mu = O(\delta^{k/2})$, so these corrections are small. 

Setting the full expressions for 
$\lambda_1^k+\lambda_2^k$ equal to $e^k-\delta^k$, and using 
the fact that $c = {\delta}/({2e-1})+O(\delta^2)$, 
gives
\[
\left ( \frac{c\Delta a}{1-c}\right)^k = -\delta^k 
- \frac{ke^{k-2}\mu^2\delta}{2e-1} + \frac{2k\mu \delta^k}{2e-1}
+ O(\mu^2\delta^2, \mu\delta^{k+1}). 
\]
Taking $k$th roots, we then have 
\begin{equation}\label{ceq1-kgon}
\frac{c \Delta a}{1-c} = -\delta + \frac{2 \mu \delta}{2e-1}
- \frac{\mu^2 e^{k-2}}{(2e-1)\delta^{k-2}} + O(\mu\delta^2, \mu^2
\delta^{3-k}),
\end{equation}
which in turn implies that 
\begin{equation}\label{ceq2-kgon} 
c = \frac{\delta + \frac{\mu^2 e^{k-2}\delta^{2-k} - 2\mu\delta}{2e-1}}{2e-1+\delta-a_1} + O(\mu\delta^2,\mu^2\delta^{3-k}).
\end{equation}

We now expand the entropy, using Taylor series around $e$ and
$1-e$, exactly as in Section \ref{sec:entropy}. As before, 
the coefficient of $H'(e)$ is $c^2(2e-1-2a_1)$. Substituting for
$c$ using (\ref{ceq2-kgon}), the leading order 
contribution of $\mu$ to our $H'(e)$ term  becomes
$$
\frac{2\mu^2 e^{k-2}\delta^{3-k}
-4 \mu \delta^2}{(2e-1)^2} H'(e). 
$$
The coefficient of $(1/2) H''(e)$ is 
$$ (2c-c^2) \left( \frac{c\Delta a}{1-c}\right)^2 - 4\mu c^2 \Delta a + c^2a_1^2,$$
exactly as before. The leading contribution of $\mu$ is the 
second term. Finally, all contributions of $\mu$ to the coefficients of higher derivatives of $H$ are higher order. Thus,
to leading order, the contribution of $\mu$ to the entropy is 
\begin{equation}\label{eq:mu-entropy}
\frac{2\mu \delta^2 H''(e)}{2e-1} - \frac{4\mu\delta^2 H'(e)}{(2e-1)^2} + \frac{2\mu^2 e^{k-2} H'(e)}{\delta^{k-3}(2e-1)^2} = \frac{4 \mu e^{k-2}\delta^{3-k}H'(e) - \delta^2 \nu}{(2e-1)^2}. 
\end{equation}
Setting the derivative with respect to $\mu$ equal to zero then
gives
\[
\mu = \frac{\nu \delta^{k-1}}{e^{k-2} H'(e)}.
\]

That is, $\mu = O(\delta^{k-1})$ and only contributes to
the entropy at order $\delta^{k+1}$. For all computations at lower
order, and in particular for computing $a$, $b$, $c$, and $d$ 
through order $\delta^{k-2}$, we can simply set $\mu=0$. That
is, to this order the computations for the edge-$k$-cycle model
are identical to the limiting problem where $\Delta d$ and 
${-c\Delta a}/({1-c})$ are set exactly equal to $\delta$. 
In that limiting problem, $a = 1-e-\delta + O(\delta^2)$. 
This then determines $c$ and $b$. The entropy then follows from
(\ref{eq:Ssubtotal}), or equivalently from 
(\ref{eq:Stotal}) with $\mu$ set equal to zero. 

Finally, we consider the Hessian of the entropy. The 
coefficient of $\mu^2$ in expression (\ref{eq:mu-entropy}) is 
large and negative, much larger than when $k=3$. However, the
linear term is the same as when $k=3$, as are the contributions
to $S$ that don't involve $\mu$. The upshot is that 
$\partial^2_{\mu\mu}S$ is more negative than when $k=3$, while
$\partial^2_{\mu a}S$ and $\partial^2_{aa}S$ are essentially the
same. Thus the Hessian is negative definite and any optimizing bipodal
graphon is unique.

\noindent{\bf Step 6:}
Starting from the bipodal graphon obtained from averaging in Step 4,
we can parametrize it in terms of $c$, $\Delta a$, and $\mu$ as in Step 5.
Then we construct a perturbation by increasing $c$ while holding $\Delta a$
and $\mu$ constant. From~\eqref{eq:k-cycle-lambda1} and~\eqref{eq:k-cycle-lambda2},
we see that
\begin{align}
    \frac{d\lambda_1}{dc} &= O(\mu^2), \\
    \frac{d\lambda_2}{dc} &= \Delta a + O(\delta).
\end{align}
It follows that
\[
    \frac{d\tau_k}{dc} = k \lambda_1^{k-1} \frac{d \lambda_1}{dc} + k \lambda_2^{k-1} \frac{d \lambda_2}{dc} =
    - k\delta^{k-1} (2e-1) + o(\delta^{k-1}),
\]
where the last equality follows from $\Delta a \approx 1-2e$, $\lambda_2 \approx -\delta$, and $\mu^2 = O(\delta^k)$.

On the other hand, the change in entropy is ${dS}/{dc} = [{2\delta}/({2e-1})]D(1-e)(1 + o(1))$, exactly as
in \Cref{sec:improvement}. That is, starting from the averaged graphon $h$ we can trade off entropy for $k$-cycles at the rate
\[
    \frac{dS}{d\tau_k} = -(1 + o(1)) \frac{2 D(1-e)}{k \delta^{k-2} (2e-1)^2} = - (1 + o(1))\frac{2 C(e)}{k \delta^{k-2}}.
\]
The rest of the proof of Theorem \ref{thm:kgon} for $\tau_k<e^k$ follows exactly as in the $k=3$ case: if the perturbation
of the averaged graphon $h$ does not contradict the optimality of $g$, both the $k$-cycle change and entropy change inequalities
of Step 4 must be sharp. It follows that $\Delta h$ approximately takes the prescribed values and that
$\Delta h \approx u_2(x) T_{\Delta h} u_2(y) + u_2(y) T_{\Delta h} u_2(x)$; but as in \Cref{sec:improvement} these two
properties contradict the definitions of $C_1$ and $C_2$ in \Cref{prop:approx-B11}.
The rest of the proof of Theorem \ref{thm:kgon} for $\tau_k<e^k$ continues as in the case $k=3$.

\medskip

We now turn to $\tau_k>e^k$. 
A graph is said to be {\em $2$-star-like} if all of its vertices 
have degree 1 or 2. In particular, all $k$-cycles are 
2-star-like. Theorem 1.1 of \cite{KRRS2} then states that 
the optimizing graphon for $\tau_k$ slightly above $e^3$ is 
bipodal with the parameters $(a,b,c,d)$ 
taking the approximate forms
indicated in Theorem \ref{thm:kgon}, only with errors that are 
$o(1)$ or $o(\Delta \tau)$ instead of the $O(\Delta \tau)$
or $O(\Delta \tau^2)$ errors claimed in Theorem \ref{thm:kgon}.

All that remains is to sharpen the estimates and compute the
entropy. We follow the same procedure as in 
Section \ref{sec:f11}, only with the expressions (\ref{eq:DeltaTau1}) and (\ref{eq:DeltaTau2}) for 
$\Delta \tau$ replaced with 
\begin{eqnarray*}
\Delta \tau_k & = & \left(\frac{c}{1-c}\right) ke^{k-2}\mu^2
+ ke^{k-3}c^2 \mu^2 (\Delta a + \Delta b - 2 \Delta d) \cr 
& & + ke^{k-4}c^2\mu^2\left (c(\Delta a-\Delta d)^2 + (1-c)(\Delta d
-\Delta b)^2 \right ) \cr 
&&+ \frac{k(k-5)e^{k-4}c^2\mu^4}{2(1-c)^2} + O(c^3).
\end{eqnarray*}
The first three terms are the multiples of $\langle 1, T_{\delta g^2} 1
\rangle$, 
$\langle 1, T_{\delta g^3} 1 \rangle$ and $\langle 1, T_{\delta g^4} 1
\rangle$, respectively, 
discussed in Step 1 of the proof for $\tau<e^k$. 
The last term involves the product of two factors of 
$\langle 1, T_{\delta g^2} 1 \rangle$ and only occurs when $k \ge 7$. 
All other
terms in the expansion of $\Delta \tau_k$, 
including $\Tr(\delta g^k)$, are of higher order. 

The upshot is that the calculation is the same as for triangles,
only with a coefficient of $ke^{k-2}$ instead of $3e$ in
the leading term, and with different $O(c^2)$ terms.
Adjusting the $O(c^2)$ terms does not affect the computation of $(a,b,c,d)$
to the order specified in Theorem \ref{thm:kgon}. 
The change from $3e$ to $ke^{k-2}$ in the leading term 
does change  the $O(\Delta \tau)$
terms in the expansions of $c$ and $b$, but does not affect
the leading expressions for $d$ or $a$, or the fact that 
the errors are indeed $O(\Delta \tau)^2)$ for $b$ and $c$
and are $O(\Delta \tau)$ for $a$ and $d$. Plugging these
values of $(a,b,c,d)$ into the formula for the entropy then
yields the estimate (\ref{eq:entropy-kgon}). 

\subsection*{Data availability statement:}
Data sharing not applicable to this article as no datasets were generated or analysed during the current study.

\end{document}